\documentclass[11pt]{amsart}

\usepackage{amsmath}
\usepackage{amssymb}
\usepackage{amscd}

\usepackage{enumitem}
\usepackage{mathtools}

\usepackage{tikz}
\usetikzlibrary{decorations.pathreplacing}
\usepackage{tikz-cd}

\usepackage{xcolor}
\newcommand{\kk}{{\mathsf{k}}}

\newcommand{\PP}{{\mathbb{P}}}

\newcommand{\ZZ}{{\mathbb{Z}}}

\newcommand{\CA}{{\mathcal{A}}}
\newcommand{\CB}{{\mathcal{B}}}
\newcommand{\CC}{{\mathcal{C}}}

\newcommand{\CE}{{\mathcal{E}}}
\newcommand{\CF}{{\mathcal{F}}}
\newcommand{\CG}{{\mathcal{G}}}
\newcommand{\CH}{{\mathcal{H}}}

\newcommand{\CO}{{\mathcal{O}}}

\newcommand{\CT}{{\mathcal{T}}}
\newcommand{\CU}{{\mathcal{U}}}
\newcommand{\CV}{{\mathcal{V}}}
\newcommand{\CW}{{\mathcal{W}}}

\newcommand{\Hom}{\mathop{\mathsf{Hom}}\nolimits}
\newcommand{\Ext}{\mathop{\mathsf{Ext}}\nolimits}

\newcommand{\Tor}{\mathop{\mathsf{Tor}}\nolimits}

\newcommand{\Cone}{\mathop{\mathsf{Cone}}\nolimits}

\newcommand{\Gr}{{\mathsf{Gr}}}
\newcommand{\OGr}{{\mathsf{OGr}}}
\newcommand{\IGr}{{\mathsf{IGr}}}
\newcommand{\LGr}{{\mathsf{LGr}}}

\newcommand{\GL}{{\mathsf{GL}}}

\newcommand{\SP}{{\mathsf{Sp}}}

\theoremstyle{plain}

\newtheorem{theorem}{Theorem}[section]
\newtheorem{conjecture}[theorem]{Conjecture}
\newtheorem{lemma}[theorem]{Lemma}
\newtheorem{proposition}[theorem]{Proposition}

\theoremstyle{definition}

\newtheorem{definition}[theorem]{Definition}

\theoremstyle{remark}

\newtheorem{remark}[theorem]{Remark}
\newtheorem{example}[theorem]{Example}

\newcommand{\You}{\mathrm{Y}}
\newcommand{\YD}{\mathrm{YD}}
\newcommand{\IFl}{{\mathsf{IFl}}}
\newcommand{\SpS}[2]{#1^{\langle #2 \rangle}}
\newcommand{\PerpQ}[1]{{#1}^\perp\!/{#1}}
\newcommand{\ptl}{\tilde{p}}
\newcommand{\qtl}{\tilde{q}}
\newcommand{\lvee}[1]{\vphantom{#1}^\vee\! {#1}}
\newcommand{\bfG}{\mathbf{G}}
\newcommand{\bfP}{\mathbf{P}}

\title{Full exceptional collections on Lagrangian Grassmannians}
\author{Anton Fonarev}
\address{\sloppy
\parbox{0.95\textwidth}{
National Research University Higher School of Economics,
Usacheva str., 6, Moscow 119048 Russia\\[.5em]
Algebraic Geometry Section, Steklov Mathematical Institute of Russian Academy of Sciences,
8~Gubkin str., Moscow 119991 Russia
\hfill
}\bigskip}
\email{avfonarev@mi-ras.ru}
\date{}
\dedicatory{To my wife Stephanie with love}
\thanks{This work was partially supported by the RSF grant 18-11-00141.}

\begin{document}

\begin{abstract}
  We show fullness of the exceptional collections of maximal length
  constructed by A.~Kuznetsov and A.~Polishchuk in the bounded
  derived categories of coherent sheaves on Lagrangian Grassmannians.
\end{abstract}

\maketitle

\section{Introduction}

Full exceptional collections proved themselves to be an incredibly useful
tool for studying derived categories of algebraic varieties.
The pioneering result in this area belongs to Beilinson,
who showed in~\cite{Bei} that the line bundles $\langle \CO, \CO(1), \ldots, \CO(n)\rangle$
form a~full exceptional collection in the bounded derived category
$D^b(\PP^n)$ of coherent sheaves on $\PP^n$.
It was later shown by Kapranov in~\cite{Kap} that the bounded derived categories of
Grassmannians, complete and partial flag varieties, and quadrics
admit full exceptional collections, consisting of equivariant vector bundles.
Since then, the following conjecture remains essentially open.

\begin{conjecture}\label{conj}
  Let $\bfG$ be a semisimple algebraic group over an algebraically closed
  field of characteristic zero, and let $\bfP\subset \bfG$
  be a parabolic subgroup. Then the bounded derived category of coherent sheaves on
  $\bfG/\bfP$ admits a full exceptional collection.
\end{conjecture}

The latter conjecture can be formulated in a stronger form. Namely, one expects to be able
to construct a full exceptional collection consisting of equivariant vector bundles.

It is not hard to reduce the general question to the case when $\bfG$ is a simple
algebraic group, and $\bfP$ is a maximal parabolic subgroup.
We refer the reader to Introduction in~\cite{KP} for a reasonably recent list
of known results in this direction. Apart from a finite number of examples,
full exceptional collections were constructed only in the bounded derived categories
of quadrics (as mentioned above) and orthogonal and isotropic Grassmannians of
planes $\IGr(2, 2n)$ and $\OGr(2, 2n+1)$ by Kuznetsov in~\cite{Kuz}.

Substantial progress was made in~\cite{KP}, where Kuznetsov and Polishchuk managed to construct exceptional
collections of maximal length (which is always equal to the rank of the Grothendieck group)
whenever $\bfG$ is a simple algebraic group of type $B$, $C$, or $D$, and $\bfP$
is maximal. The method they used is quite curious. First, they observe that the equivariant
derived category admits an infinite full exceptional collection; namely, one can simply take
all the irreducible equivariant vector bundles.
Next, they suggest a representation-theoretic criterion under which the dual (in the equivariant
category) to a~finite subcollection of such bundles is exceptional in the non-equivariant
derived category. The~authors call such a subcollection an \emph{exceptional block}. Finally,
they do a case-by-case study showing how to~choose exceptional blocks (several for each variety)
so that the objects coming from different blocks satisfy semiorthogonality conditions,
and their number equals the rank of the Grothendieck group.

Since one expects that in the case of rational homogeneous varieties (more generally,
in any triangulated category generated by a full exceptional collection)
any exceptional collection of maximal length is full, a natural approach to
Conjecture~\ref{conj}
is to show fullness of the exceptional collections of Kuznetsov and Polishchuk.
Unfortunately, the task is not that easy: the exceptional objects are constructed in quite
an abstract way (it is not even clear whether the collections consist of coherent sheaves),
while no general method of showing fullness of a given exceptional collection is known
(the resolution of diagonal method, invented by Beilinson and used by Kapranov, does not
easily apply for general isotropic and orthogonal Grassmannians).

The purpose of the present work is to give an explicit geometric description of the exceptional
objects of Kuznetsov and Polishchuk in the case of Lagrangian Grassmannians $\LGr(n, 2n)$
and to show that the corresponding collections are full. In order to do the latter,
we construct a certain class of exact complexes, which we call
\emph{Lagrangian staircase complexes}. Staircase complexes appeared in~\cite{Fon-LD},
and were used to construct certain Lefschetz decompositions of the derived categories
of the usual Grassmannians $\Gr(k, n)$. Their generalization later appeared in~\cite{Fon-KP},
where the construction of Kuznetsov and Polishchuk was studied in type $A$.

One way to show that a given exceptional collection in the bounded derived category of
a smooth projective variety $X$ is full is to show that the subcategory generated by this
collection contains $\CO_X$ and is stable under the twist by an ample line bundle $\CO_X(1)$.
This is where staircase complexes turn out to be very useful.

The main result of the paper is the following theorem.

\begin{theorem}\label{thm:meta}
  Let $V$ be a $2n$-dimensional symplectic vector space over an algebraically closed field
  of characteristic zero.
  Then the bounded derived category $D^b(\LGr(n, V))$ of coherent sheaves on $\LGr(n, V)$
  admits a full exceptional collection consisting of equivariant vector bundles.
\end{theorem}

Theorem~\ref{thm:meta} was previously shown to hold for $\LGr(3, 6)$ by Samokhin, see~\cite{Sam},
and for $\LGr(4, 8)$ by Polishchuk and Samokhin, see~\cite{Sam-Pol}. The latter paper also
dealt with the derived category of $\LGr(5, 10)$, see Remark~\ref{rm:pol-sam}.

The paper is organized as follows. In Section~2 we collect some preliminaries.
In Section~3 we give two different geometric descriptions of the Kuznetsov--Polishchuk
objects on Lagrangian Grassmannians. Both descriptions are essential in the construction
of Lagrangian staircase complexes, which is done in Section~4. The proof of the main
theorem concludes the latter. In Section~5 we use staircase complexes to construct
a minimal Lefschetz exceptional collection in the derived category of $D^b(\LGr(5, 10))$.
We tried to be kind to the reader and pulled all the Borel--Bott--Weil computations in Appendix~A.

\subsection*{Acknowledgements}
The existence of Lagrangian staircase complexes was predicted many years ago
by A.~Kuznetsov. I am grateful to him for his mathematical generosity and
genuine interest in the present work.

\section{Preliminaries}

We work over a fixed algebraically closed field $\kk$ of characteristic zero.

\subsection{Semiorthogonal decompositions and exceptional collections}\label{ssec:exc}

We will freely use the notions
of an exceptional collection and semiorthogonality decomposition. For convenience,
we remind the reader of the following facts.

Let $\CT$ be a $\kk$-linear triangulated category.

\begin{definition}
A full subcategory $\CA\subseteq\CT$ is called \emph{admissible}
if the inclusion functor has both a left and a right adjoint.
\end{definition}

With every admissible subcategory $\CA\subseteq\CT$ one can associate
two semiorthogonal decompositions:
\begin{equation*}
  \CT = \langle \CA^\perp, \CA\rangle, \quad \text{where} \quad
  \CA^\perp = \langle X\in \CT\mid \Hom(\CA, X)=0\rangle,
\end{equation*}
and
\begin{equation*}
  \CT = \langle \CA, \vphantom{\CA}^\perp\!\CA\rangle, \quad \text{where} \quad
  \vphantom{\CA}^\perp\!\CA = \langle Y\in \CT\mid \Hom(Y,\CA)=0\rangle.
\end{equation*}
Recall that every full triangulated subcategory
generated by a full exceptional collection is admissible if the ambient category
is triangulated.

Let $\CT=\langle \CA_1, \CA_2,\ldots, \CA_t\rangle$ be
a semiorthogonal decomposition.
Then for every object $X\in \CT$ there exists a functorial filtration
\begin{equation*}
  0=X_t\to X_{t-1}\to \cdots\to X_0=X,
\end{equation*}
such that for every $i=1,\ldots, t$ the cone $Y_i$ of the corresponding morphism
\begin{equation*}
  X_{i}\to X_{i-1}\to Y_i\to X_{i}[1]
\end{equation*}
belongs to $\CA_i$.

Let $\CT$ be a $\kk$-linear triangulated category,
and let $\langle E_1, E_2,\ldots, E_t\rangle$ be an exceptional collection.
Then one can construct two more exceptional collections in $\CT$; namely, the left and right
dual exceptional collections $\langle E^\vee_1, E^\vee_2,\ldots, E^\vee_t\rangle$ and
$\langle \lvee{E}_1, \lvee{E}_2,\ldots, \lvee{E}_t\rangle$. The dual collections can
characterized by the following properties. First,
$E^\vee_i\in \langle E_1, E_2,\ldots, E_i\rangle$ and
$\lvee{E}_i\in \langle E_{t-i+1}, E_{t-i+2},\ldots, E_t\rangle$ for all $i=1,\ldots,n$.
Next,
\begin{equation*}
  \Hom^\bullet (E_i, E^\vee_j) = \begin{cases}
    \kk & \text{if } i+j = n+1, \\
    0 & \text{otherwise,}
  \end{cases}
  \quad\text{and}\quad
  \Hom^\bullet (\lvee{E}_i, E_j) = \begin{cases}
    \kk & \text{if } i+j = n+1, \\
    0 & \text{otherwise.}
  \end{cases}
\end{equation*}
It follows from the definition that the left (resp. right) dual collection of
the right (resp. left) dual collection is isomorphic to the original collection.

\begin{remark}\label{rm:dual}
Note that we did not specify the cohomological degrees in which the nontrivial
morphisms between objects of the collection and its duals collections live.
Moreover, various conventions appear in the literature.
Different choices lead to exceptional collections whose objects only differ by
shifts in the~triangulated category. Since, the associated semiorthogonal
decompositions are identical, there is some freedom in the choice.
On all occasions we choose the degrees so that the dual collection
of interest consists of vector bundles.
\end{remark}

\subsection{Weights and diagrams}
Let $k$ be a positive integer.
Denote by $\You_k\subset \ZZ^k$ the set of weakly decreasing
sequences
$\You_k=\left\{(\lambda_1, \lambda_2, \ldots, \lambda_k)\in\ZZ^k \mid
\lambda_1\geq \lambda_2\geq \cdots\geq \lambda_k\right\}$. We sometimes refer to
elements of $\You_k$ as to \emph{weights} since the set $\You_k$ can be
naturally identified with the set of dominant weights of the group $\GL_k$.

By a \emph{Young diagram} we mean a weight with nonnegative terms.
The set of Young diagrams is denoted by
$\YD_k=\left\{(\lambda_1, \lambda_2, \ldots, \lambda_k)\in\ZZ^k \mid
\lambda_1\geq \lambda_2\geq \cdots\geq \lambda_k\geq 0\right\} \subset \You_k$.
Given a Young diagram $\lambda$, we denote
by $|\lambda|=\lambda_1+\lambda_2+\cdots+\lambda_k$ its size, and
by $\lambda^T$ its transpose: $\lambda^T\in \YD_{\lambda_1}$,
and ${\lambda^T}_i = \max \{1\leq j \leq k\mid \lambda_j\geq i\}$.
If $\lambda\in \YD_k$ is a Young diagram, we can naturally treat
it as an element in $\YD_l$ for any $l\geq k$, just by extending
the corresponding sequence with zeros.

There is a natural inclusion partial order on $\You_k$; namely,
\begin{equation*}
  \lambda\subseteq\mu \quad \Leftrightarrow \quad \lambda_i\leq\mu_i \text{ for all } i=1,\ldots,k.
\end{equation*}
When applied to Young diagrams, $\lambda\subseteq \mu$ just means that the
diagram $\lambda$ fits into the diagram $\mu$.

In the following we will work with some specific subsets of weights.
For a given pair of non-negative integers $h$ and $w$,
let $\You_{h,w}\subset \YD_h$ denote the set of those Young diagrams
whose width is at most $w$:
\begin{equation*}
  \You_{h,w}=\left\{(\lambda_1, \lambda_2, \ldots, \lambda_h)\in\ZZ^h \mid
w\geq \lambda_1\geq \lambda_2\geq \cdots\geq \lambda_h\geq 0\right\}.
\end{equation*}
It is easy to see that transposition provides a bijection between
$\You_{h,w}$ and $\You_{w, h}$. The set $\You_{h,w}$ is naturally in bijection
with the set of binary sequences of length $h+w$ containing
exactly $h$ zeros. We describe the map from the latter set to the former.
Let $\bar{a} = a_1a_2\cdots a_{h+w}\in \{0,1\}^{h+w}$ be such a sequence, and
let $1\leq l_1<l_2<\ldots<l_h\leq h+w$ be all the indices for which
$a_{l_j}=0$. Then with $\bar{a}$ we associate the diagram
\begin{equation*}
  (l_k-k,\ l_{k-1}-(k-1),\ \ldots, l_2-2,\ l_1-1)\in \You_{h,w}.
\end{equation*}

Last but not least, we will use a couple of group actions on the set $\You_h$.
The first one is the involution on the set $\You_h$, which sends $\lambda\in\You_h$
to
\begin{equation*}
  -\lambda = (-\lambda_h, \ldots, -\lambda_2, -\lambda_1).
\end{equation*}
The second one is the action of the group $\ZZ$ given by
\begin{equation}\label{eq:tw}
  \lambda(t) = (\lambda_1+t, \lambda_2+t, \ldots, \lambda_1+t).
\end{equation}
The two actions combined induce an action of the group $\ZZ\rtimes \ZZ/(2)$.

\subsection{Equivariant vector bundles}

As we have already mentioned, the set $\You_k$ can be naturally identified
with the set of dominant weights of the group $\GL_k$.
Given a rank $k$ vector bundle $\CU$ on a scheme $X$ and a weight $\lambda\in\You_k$,
we denote by
$\Sigma^\lambda\CU$ the vector bundle associated with the irreducible
$\GL_k$ representation of highest weight $\lambda$ and the principal $\GL_k$-bundle
associated with $\CU$.

If $\lambda$ is a Young diagram,
$\Sigma^\lambda$ is the usual Schur functor. In particular, if the number of
non-zero rows in~$\lambda$ is greater than $k$, then $\Sigma^\lambda\CU=0$.
Our convention is that, $\Sigma^{(i, 0, \ldots, 0)}\CU \simeq S^i\CU$,
and $\Sigma^{(1,\ldots, 1, 0, \ldots, 0)}\CU \simeq \Lambda^t\CU$,
where $t$ is the number of nonzero rows in the corresponding diagram.
We will often use the standard isomorphisms
\begin{equation*}
  \Sigma^{-\lambda}\CU\simeq \Sigma^\lambda\CU^*
  \quad\text{and}\quad
  \Sigma^{\lambda(t)}\CU\simeq \Sigma^\lambda\CU\otimes (\det \CU)^{\otimes t}.
\end{equation*}

Given a pair of weights $\lambda,\mu\in\You_k$, the tensor product
$\Sigma^\lambda\CU\otimes\Sigma^\mu\CU$ can be decomposed
into a direct sum of bundles of the form $\Sigma^\nu\CU$ (the irreducible
summands), using the Littlewood--Richardson rule: there is an isomorphism
of vector bundles
\begin{equation}\label{eq:lrr}
  \Sigma^\lambda\CU\otimes\Sigma^\mu\CU \simeq \bigoplus \left(\Sigma^\nu\CU\right)^{\oplus c^\nu_{\lambda,\mu}},
\end{equation}
where the numbers $c^\nu_{\lambda,\mu}$ are called the Littlewood--Richardson
coefficients.

We refer the reader to the wonderful book~\cite{Ful} for details.
We will need the following two easy statements, which follow immediately
from this rule.

\begin{lemma}\label{lm:prod}
  Let $\CU$ be a rank $h$ vector bundle on a scheme $X$, and let $\lambda,\mu\in\YD_h$
  be two Young diagrams.
  Then for every irreducible summand
  $\Sigma^\nu\CU\subseteq \Sigma^\lambda\CU\otimes\Sigma^\mu\CU^*$ one has
  \begin{equation*}
    -\mu\subseteq \nu\subseteq \lambda.
  \end{equation*}
\end{lemma}

\begin{lemma}\label{lm:pos-sum}
  Let $\CU$ be a rank $h$ vector bundle on a scheme $X$, and let $\lambda,\mu\in\YD_h$
  be two Young diagrams.
  If there is an irreducible summand
  $\Sigma^\nu\CU\subseteq \Sigma^\lambda\CU\otimes\Sigma^\mu\CU^*$ such that
  $\nu$ is a Young diagram (that is, $\nu_h\geq 0$), then $\mu\subseteq \lambda$.
\end{lemma}

Recall that to a pair of Young diagrams $\mu\subseteq\lambda$ one can associate
the so-called skew Schur functor $\Sigma^{\lambda/\mu}$, which satisfies the property
\[
  \Sigma^{\lambda/\mu}\CU \simeq \bigoplus \left(\Sigma^\nu\CU\right)^{\oplus c^\lambda_{\nu,\mu}},
\]
where $c^\lambda_{\nu,\mu}$ are Littlewood--Richardson coefficients appearing in~\eqref{eq:lrr}.
Skew Schur functors are particularly useful to us because of the following result.

\begin{lemma}[{\cite[Proposition~2.3.1]{Wey}}]\label{lm:sch-ses}
  Let $\lambda$ be a Young diagram, and let
  \[
    0\to \CU\to \CF\to \CG\to 0
  \]
  be a short exact sequence of vector bundles on a scheme $X$.
  The exists a filtration on $\Sigma^\lambda\CF$ with the associated graded
  isomorphic to
  \[
    \bigoplus_{\mu\subseteq \lambda} \Sigma^\mu \CU\otimes\Sigma^{\lambda/\mu}\CG.
  \]
  Similarly, there exists a filtration on $\Sigma^\lambda\CF$ with the associated
  graded isomorphic to
  \[
    \bigoplus_{\mu\subseteq \lambda} \Sigma^{\lambda/\mu}\CU\otimes\Sigma^\mu\CG.
  \]
\end{lemma}

\subsection{Isotropic Grassmannians and symplectic Schur functors}
By a symplectic vector bundle on a scheme $X$ we mean a locally free sheaf $\CV$
together with a section $\CO_X\to \Lambda^2\CV^*$ such that the associated
morphism $\CV\to \CV^*$ is a skew-symmetric isomorphism.

Recall that the set of dominant weights
of the group $\SP_{2n}$ is naturally identified with the set $\YD_n$.
Given a $2n$-dimensional symplectic vector space $V$ (resp. symplectic vector
bundle $\CV$), we denote by $\SpS{V}{\lambda}$ (resp. $\SpS{\CV}{\lambda}$)
the result of the application of the corresponding symplectic Schur functor;
that is, the vector bundle associated with the highest weight $\lambda$ and
the principal $\SP_{2n}$-bundle associated with $\CV$.
Recall that $\SpS{V}{\lambda}$ (resp. $\SpS{\CV}{\lambda}$) is a quotient
of $\Sigma^\lambda V$ (resp. $\Sigma^\lambda \CV$).

Let $V$ be a $2n$-dimensional symplectic vector space. We denote by $\IGr(k, V)$
the isotropic Grassmannian, which parametrizes $k$-dimensional isotropic
subspaces. When $k=n$, we get the Lagrangian Grassmannian $\LGr(n, V)$.
The varieties $\IGr(k, V)$ are precisely the rational homogeneous varieties for
the group $\bfG=\SP_{2n}$ and maximal parabolic subgroups $\bfP$.
It is well known that the irreducible $\bfG$-equivariant bundles on $\bfG/\bfP$
are parametrized by the dominant weights of the Levi quotient of $\bfP$.
In our case one can easily describe them.

Denote by $\CU$ the tautological rank $k$ vector bundle on $\IGr(k, V)$,
and by $\CU^\perp$ the rank $2n-k$ subbundle given by vectors orthogonal to $\CU$
with respect to the symplectic form on $V$. The isotropic condition provides
an inclusion $\CU\subseteq \CU^\perp$.
The symplectic structure on $V$ descends to $\PerpQ{\CU}$,
and every irreducible $\bfG$-equivariant vector bundle on $\IGr(k, V)$
is of the form $\Sigma^\lambda \CU \otimes \SpS{(\PerpQ{\CU})}{\mu}$ for some
$\lambda\in Y_{k}$ and $\mu\in \YD_{n-k}$.

The isotropic Grassmannian is naturally embedded in the usual Grassmannian
$\Gr(k, V)$ as a closed subvariety, and the tautological vector bundle on $\IGr(k, V)$
is the restriction of the tautological bundle on $\Gr(k, V)$. Moreover, the restriction
of $(V/\CU)^*$ from the Grassmannian to the isotropic Grassmannian is naturally
isomorphic to $\CU^\perp$. In the case of $\LGr(n, V)$ we have $\CU\simeq \CU^\perp$.

Of course, isotropic and Lagrangian Grassmannians exist in the relative setting.
The following lemma is trivial, we include its proof for the reader's convenience.

\begin{lemma}\label{lm:sp-filtration}
  Let $\CV$ be a symplectic vector bundle of rank $2n$ on a smooth algebraic variety $X$.
  Consider the relative tautological Grassmannian together with the natural projection
  \begin{equation*}
    p: \LGr_X(n, \CV) \to X,
  \end{equation*}
  and denote by $\CU$ the relative Lagrangian subbundle on $\LGr_X(n,\CV)$.
  Let $\lambda\in \YD_n$ be a Young diagram.
  Then $p^*\SpS{\CV}{\lambda}$ admits a filtration
  $p^*\SpS{\CV}{\lambda}=\CV_N\supseteq\CV_{N-1}\supseteq\cdots\supseteq\CV_0=0$
  such that the associated quotients $\CV_i/\CV_{i-1}$ are of the form
  $\Sigma^{\mu_i}\CU$ for some $-\lambda\subseteq \mu_i\subseteq \lambda$.
\end{lemma}
\begin{proof}
  Recall that there is a closed embedding $\iota: \LGr_X(n, \CV)\to \Gr_X(n, \CV)$,
  where $\Gr_X(n, \CV)$ denotes the relative Grassmannian. Consider the diagram
  \begin{equation*}
    \begin{tikzcd}[column sep=small]
      \LGr_X(n, \CV) \arrow[rr, "\iota"] \arrow[rd, "p" swap] & & \Gr_X(n, \CV) \arrow[ld, "q"] \\
      & X & 
    \end{tikzcd}.
  \end{equation*}
  Without creating any confusion, denote by $\CU$ the relative tautological bundle
  on $\Gr_X(n, \CV)$.
  Recall that $\SpS{\CV}{\lambda}$ is a quotient of $\Sigma^\lambda\CV$. In particular,
  $p^*\SpS{\CV}{\lambda}$ is a quotient of $p^*\Sigma^\lambda\CV=\iota^*\Sigma^\lambda (q^*\CV)$.
  Let us apply Lemma~\ref{lm:sch-ses} to the short exact sequence
  \begin{equation*}
    0\to \CU \to q^*\CV\to q^*\CV/\CU\to 0
  \end{equation*}
  of vector bundles on $\Gr(n, \CV)$. We get a filtration with associated subquotients
  of the form $\Sigma^\mu\CU\otimes \Sigma^\nu (\CV/\CU)$ (we used the skew Schur
  functor decomposition property), where $\mu, \nu\in \YD_n$ and $\mu,\nu\subseteq \lambda$.
  Since $\iota^*(\CV/\CU)\simeq \CU^*$, the result follows from Lemma~\ref{lm:prod}.
\end{proof}

\subsection{Lagrangian exceptional blocks of Kuznetsov--Polishchuk}
We are now going to sketch the results of~\cite{KP} in the case of Lagrangian Grassmannians.
Let $V$ be a $2n$-dimensional symplectic vector space, and let $\bfG=\SP(V)$.
We are interested in the derived category of $\LGr(n, V)$. As usual, let
$\CU$ denote the tautological rank $n$ vector bundle on $\LGr(n, V)$.
Since the symplectic form induces an isomorphism $V/\CU\simeq \CU^*$,
the tautological short exact sequence on $\LGr(n, V)$ is of the form
\[
  0\to \CU \to V \to \CU^*\to 0.
\]

It was explained before that irreducible $\bfG$-equivariant vector bundles on $\LGr(n, V)$
are all of the form $\Sigma^\lambda\CU^*$, where $\lambda\in \You_n$
(the Levi quotient of the corresponding maximal parabolic subgroup is isomorphic
to $\GL_n$). Moreover, they form an infinite full exceptional collection in the
equivariant bounded derived category $D^b_{\bfG}(\LGr(n, V))$.

\begin{definition}[{\cite[Definition~3.1]{KP}}]
  A subset of weights $B\subset \You_n$ is called an \emph{exceptional block}
  if for any $\lambda,\mu\in B$ the canonical map
  \begin{equation*}\label{eq:exc-blk}
    \bigoplus_{\nu\in B} \Ext^\bullet_{\bfG}(\Sigma^\lambda\CU^*, \Sigma^\nu\CU^*)
    \otimes \Hom(\Sigma^\nu\CU^*, \Sigma^\mu\CU^*) \to
    \Ext^\bullet(\Sigma^\lambda\CU^*, \Sigma^\mu\CU^*)
  \end{equation*}
  is an isomorphism.
\end{definition}

Kuznetsov and Polishchuk made the following wonderful observation.

\begin{proposition}[{\cite[Proposition~3.9]{KP}}]\label{prop:blk}
  Given an exceptional block $B$, let $\langle \CE^\lambda \mid \lambda\in B\rangle$
  denote the right dual exceptional collection to $\langle \Sigma^\lambda\CU^* \mid \lambda\in B\rangle$
  in the equivariant derived category $D^b_{\bfG}(\LGr(n, V))$. Then
  $\langle \CE^\lambda \mid \lambda\in B\rangle$ form an exceptional collection
  in the non-equivariant derived category $D^b(\LGr(n, V))$.
\end{proposition}

Various exceptional blocks were constructed for orthogonal and isotropic Grassmannians
in~\cite{KP}. We are mainly interested in the case of Lagrangian Grassmannians.
Let $h, w\geq 0$ be integers such that $h+w\leq n+1$. It was shown in~\cite[Section~5]{KP}
that the set of weights $\You_{h, w}\subset \You_n$ forms an exceptional block.
Using the previous proposition, one can construct exceptional objects
$\CE^\lambda\in D^b(\LGr(n, V))$ for all $\lambda\in \You_n$ such that
$\lambda\in \You_{h,w}$ for some integers $h$ and $w$ such that $h+w\leq n+1$.
An attentive reader might point out that the notation for $\CE^\lambda$ does not
reflect the choice of an exceptional block. In fact, there is
no dependence on such a choice. Let us fix integers $h,w\geq 0$ such that
$h+w\leq n+1$, and let $\lambda\in \You_{h,w}$.
It easily follows from the Borel--Bott--Weil
theorem that
\begin{equation*}
  \Ext^\bullet_{\bfG}(\Sigma^\nu\CU^*, \Sigma^\mu\CU^*) \simeq
  \Ext^\bullet(\Sigma^\nu\CU^*, \Sigma^\mu\CU^*)^\bfG = 0
  \quad \text{if} \quad \mu\nsubseteq \nu.
\end{equation*}
In particular, we can order the exceptional collection
\begin{equation}\label{eq:blk-hw}
  \langle \Sigma^\mu\CU^* \mid \mu\in \You_{h,w}\rangle
\end{equation}
in $D^b_{\bfG}(\LGr(n, V))$
so that $\langle \Sigma^\mu\CU^* \mid \mu\in \You_{h,w}, \mu\subseteq \lambda\rangle$
are the rightmost objects in~\eqref{eq:blk-hw}. It now follows from our discussion
of dual exceptional collections in Section~\ref{ssec:exc} that $\CE^\lambda$ can be
(up to isomorphism) characterized by the following properties:
\begin{equation}\label{eq:el-up}
  \CE^\lambda\in \langle \Sigma^\mu\CU^* \mid 0\subseteq \mu\subseteq \lambda\rangle
  \subset D^b_{\bfG}(\LGr(n, V))\quad \text{and} \quad
  \Ext^\bullet_{\bfG}(\CE^\lambda, \Sigma^\mu\CU^*) = \begin{cases}
    \kk & \text{if } \mu = \lambda, \\
    0 & \text{if } 0\subseteq \mu \subsetneq \lambda.
  \end{cases}
\end{equation}
In particular, $\CE^\lambda$ does not depend on the choice of an exceptional block.

Given a full triangulated subcategory $\CC\subseteq D^b(\LGr(n, V))$,
we denote by $\CC(i)$ the image of $\CC$
under the autoequivalence given by $-\otimes\CO(i)$.
Recall that $\YD_k$ can be naturally considered as a subset in $\YD_n$ for all $k\leq n$.

\begin{proposition}[{\cite[Theorem~9.2]{KP}}]\label{prop:fec-kp}
  There is a semiorthogonal decomposition
  \begin{equation}\label{eq:kp-coll}
    \left\langle \CC_0, \CC_1(1), \CC_2(2), \ldots, \CC_n(n) \right\rangle \subseteq D^b(\LGr(n, V)),
    \quad\text{where}\quad
    \CC_i = \left\langle \CE^\lambda \mid \lambda\in\You_{i, n-i}  \right\rangle.
  \end{equation}
\end{proposition}

\begin{remark}
  In the previous proposition we did not specify how exceptional objects
  are ordered within each block of the semiorthogonal decomposition.
  One can pick any total ordering of $\CE^\lambda$ in
  $\left\langle \CE^\lambda \mid \lambda\in\You_{h, w}  \right\rangle$
  refining the partial order $\subseteq$ on $\You_{h, w}$.
\end{remark}

\begin{remark}
  Proposition~\ref{prop:fec-kp} only deals with blocks of the form $\You_{h, w}$,
  where $h+w=n$, while we defined exceptional objects $\CE^\lambda$ for
  $\lambda\in\You_{h,w}$ with $h+w=n+1$ as well.
  These extra objects will appear in the proof of fullness of the
  exceptional collection given by~\eqref{eq:kp-coll}.
\end{remark}

It will be more convenient for us to work with objects dual to $\CE^\lambda$,
which we denote by $\CF^\lambda$:
\[
  \CF^\lambda = (\CE^\lambda)^*.
\]
Since duality is an anti-autoequivalence,
right dual exceptional collections become left dual, and conditions~\eqref{eq:el-up}
translate to the following characterization of $\CF^\lambda$:
\begin{equation}\label{eq:fl-up}
  \CF^\lambda\in \langle \Sigma^\mu\CU \mid 0\subseteq \mu\subseteq \lambda\rangle
  \subset D^b_{\bfG}(\LGr(n, V))\quad \text{and} \quad
  \Ext^\bullet_{\bfG}(\Sigma^\mu\CU, \CF^\lambda) = \begin{cases}
    \kk & \text{if } \mu = \lambda, \\
    0 & \text{if } 0\subseteq \mu \subsetneq \lambda.
  \end{cases}
\end{equation}

\begin{remark}
  One can say that a subset of weights $B\subset \You_n$ is a \emph{left exceptional block}
  if it satisfies the condition
  \begin{equation*}
    \bigoplus_{\nu\in B} \Hom(\Sigma^\lambda\CU^*, \Sigma^\nu\CU^*)\otimes
    \Ext^\bullet_{\bfG}(\Sigma^\nu\CU^*, \Sigma^\mu\CU^*)
    \to
    \Ext^\bullet(\Sigma^\lambda\CU^*, \Sigma^\mu\CU^*).
  \end{equation*}
  The same argument which was used in~\cite{KP} to prove Proposition~\ref{prop:blk} shows
  that given a left exceptional block $B$, the \emph{left} dual exceptional collection
  to $\langle \Sigma^\mu\CU^* \mid \mu\in \You_{h,w}\rangle$ in $D^b_{\bfG}(\LGr(n, V))$
  forms an exceptional collection in $D^b(\LGr(n, V))$.
  Since duality translates to negation of weights, we conclude that
  $B$ is a left exceptional block if and only if the set of weights
  $-B=\left\{-\lambda\ |\ \lambda\in B\right\}$ is an exceptional block.
\end{remark}

\section{Exceptional objects}
Let us fix positive integers $h$ and $w$ such that $n\geq h,w\geq 1$ and $h+w=n+1$.
Our goal is to give two descriptions of the objects $\CF^\lambda$ for
$\lambda\in\You_{h, w}$. The first description is more geometric. It~expresses
these objects as pushforwards of some equivariant irreducible vector bundles
on partial flag varieties. The~second one is slightly less trivial: it relates
$\CF^\lambda$ to certain exceptional objects on isotropic Grassmannians parametrizing
subspaces of smaller dimension.

\subsection{First description}
We have to deal with two separate cases. We begin with the case when $\lambda_h=0$.

Consider the partial flag variety $\IFl(w, n; V)$ together with the two projection maps.
\begin{equation}\label{eq:cd-pq}
\begin{tikzcd}
  & \IFl(w, n; V) \arrow[ld, "p" swap] \arrow[rd, "q"] & \\
  \LGr(n, V) & & \IGr(w, V).
\end{tikzcd}
\end{equation}
Denote by $\CU$ and $\CW$ the tautological bundles on $\LGr(n, V)$ and on $\IGr(w, V)$ respectively
as well as their pullbacks on $\IFl(w, n; V)$. Remark that $\PerpQ{\CW}$ is a symplectic
vector bundle of rank $2(n-w)=2(h-1)$. The projection $p$ realizes $\IFl(w, n; V)$
as the relative Grassmannian $\Gr(w, \CU)$, while the projection $q$ realizes $\IFl(w, n; V)$
as the relative Lagrangian Grassmannian $\LGr(h-1, \PerpQ{\CW})$.
Remark that under our assumptions the diagram $\lambda$ has at most $h-1$
nonzero rows, and thus the bundle $\SpS{\left(\PerpQ{\CW}\right)}{\lambda}$ is nonzero.

\begin{proposition}\label{prop:fl-reg-sp}
  The exceptional object $\CF^\lambda$ is isomorphic to $p_*q^*\SpS{\left(\PerpQ{\CW}\right)}{\lambda}$.
\end{proposition}
\begin{proof}
  Put $\kappa=(\lambda_1, \lambda_2, \ldots, \lambda_{h-1})\in \You_{h-1, w}$.
  Since both the symplectic Schur functor and $\CF^\lambda$ depend only on the Young diagram
  shape, and $\kappa$ is obtained from $\lambda$ by dropping an empty row,
  it is enough to show that $\CF^{\kappa}\simeq p_*q^*\SpS{\left(\PerpQ{\CW}\right)}{\kappa}$.
  In order to construct the required isomorphism, we will check that the bundle
  $\CF=p_*q^*\SpS{\left(\PerpQ{\CW}\right)}{\kappa}$ satisfies the dual exceptional
  collection condition~\eqref{eq:fl-up}: the object $\CF$ belongs to
  $\langle\, \Sigma^\mu\CU \mid 0\subseteq \mu\subseteq\kappa \rangle$
  in the equivariant derived category, and for any diagram $\mu\subseteq\kappa$
  \begin{equation*}
    \Hom_\bfG^\bullet(\Sigma^\mu\CU, \CF) =
    \begin{cases}
      \kk & \quad \text{if } \mu=\kappa, \\
      0   & \quad \text{otherwise.}
    \end{cases}
  \end{equation*}

  We first check the containment condition.
  By Lemma~\ref{lm:sp-filtration}, the bundle
  $\SpS{\left(\PerpQ{\CW}\right)}{\kappa}$ is an~iterated extension of equivariant vector bundles
  of the form
  $\Sigma^\nu\!\left(\CU/\CW\right)$, where $-\kappa\subseteq \nu \subseteq \kappa$.
  In particular, $\nu_h\geq -\kappa_1\geq -w$. The projection $p$ is nothing but the
  relative Grassmannian $\Gr(w, \CU)$. Using Lemma~\ref{lm:gr-0},
  we see that
  \begin{equation*}
    R^ip_*\Sigma^\nu\!\left(\CU/\CW\right) = 
    \begin{cases}
      \CU^\nu & \quad \text{if } \nu \supseteq 0 \text{ and } i = 0, \\
      0       & \quad \text{otherwise}.
    \end{cases}
  \end{equation*}
  Since the projection $p$ is an equivariant morphism,
  from the spectral sequence associated with the latter filtration we see that
  $\CF$ is an equivariant iterated extension of vector bundles of the form
  $\Sigma^\nu\CU$, where $0\subseteq \nu \subseteq \kappa=\lambda$.

  Now we compute necessary $\Hom$ groups.
  As both projections $p$ and $q$ are $\bfG$-equivariant,
  \begin{align}\label{eq:fl1}
    \Hom_\bfG^\bullet(\Sigma^\mu\CU,\,\CF)
    & \simeq \Hom_{\bfG}^\bullet(\Sigma^\mu\CU,\,p_*q^*{\SpS{(\PerpQ{\CW})}{\kappa}}) \nonumber \\ 
    & \simeq \Hom_\bfG^\bullet(\Sigma^\mu\CU,\,q^*\SpS{(\PerpQ{\CW})}{\kappa}) \nonumber \\
    & \simeq H^\bullet(\IFl(w, n; V),\,q^*\SpS{(\PerpQ{\CW})}{\kappa} \otimes \Sigma^\mu\CU^*)^\bfG \nonumber \\
    & \simeq H^\bullet(\IGr(w, V),\,\SpS{(\PerpQ{\CW})}{\kappa}\otimes q_*\Sigma^\mu\CU^*)^\bfG.
  \end{align}
  Consider the short exact sequence of vector bundles on $\IFl(w, n; V)$
  \begin{equation*}
    0\to (\CU/\CW)^*\to \CU^*\to \CW^*\to 0.
  \end{equation*} 
  By Lemma~\ref{lm:sch-ses}, there is a filtration on $\Sigma^\mu\CU^*$ with
  the associated graded of the form
  \begin{equation*}
    \bigoplus_{0\subseteq\nu\subseteq\mu} \Sigma^\nu (\CU/\CW)^* \otimes \Sigma^{\mu/\nu} \CW^*.
  \end{equation*}
  Using Lemma~\ref{lm:lgr-0} and the projection formula for the second projection $q$,
  one gets a spectral sequence whose terms are of the form
  \begin{align}\label{eq:fl2}
    & H^\bullet(\IGr(w, V),\, \SpS{(\PerpQ{\CW})}{\kappa}\otimes q_*\Sigma^\nu (\CU/\CW)^* \otimes \Sigma^{\mu/\nu} \CW^*)^\bfG \nonumber \\
    & = H^\bullet(\IGr(w, V),\, \SpS{(\PerpQ{\CW})}{\kappa}\otimes \SpS{(\PerpQ{\CW})}{\nu} \otimes \Sigma^{\mu/\nu} \CW^*)^\bfG
  \end{align}
  converging to~\eqref{eq:fl1}.
  Remark that $\SpS{(\PerpQ{\CW})}{\kappa} \otimes \Sigma^{\mu/\nu} \CW^*\otimes \SpS{(\PerpQ{\CW})}{\nu}$
  splits into a direct sum of irreducible equivariant bundles of the form
  $\SpS{(\PerpQ{\CW})}{\alpha}\otimes \Sigma^\beta\CW^*$, where
  $\SpS{(\PerpQ{\CW})}{\alpha}\subseteq \SpS{(\PerpQ{\CW})}{\kappa}\otimes \SpS{(\PerpQ{\CW})}{\nu}$,
  and $\Sigma^\beta \CW^* \subseteq \Sigma^{\mu/\nu} \CW^*$. From Lemma~\ref{lm:igr-0}
  we know that
  \begin{equation*}
    H^\bullet(\IGr(w, V),\,\SpS{(\PerpQ{\CW})}{\alpha}\otimes \Sigma^\beta\CW^*)^{\bfG} = 0
  \end{equation*}
  as soon as $\beta > 0$, while $\beta=0$ is only possible when $\nu=\mu$.
  If $\nu=\mu$, then
  \begin{align*}
    H^\bullet(\IGr(w, V),\,\SpS{(\PerpQ{\CW})}{\kappa} \otimes \Sigma^{\mu/\nu} \CW^*\otimes \SpS{(\PerpQ{\CW})}{\nu})^\bfG
    & = H^\bullet(\IGr(w, V),\,\SpS{(\PerpQ{\CW})}{\kappa} \otimes \SpS{(\PerpQ{\CW})}{\mu})^\bfG \\
    & = \Ext^\bullet_\bfG(\SpS{(\PerpQ{\CW})}{\kappa},\,\SpS{(\PerpQ{\CW})}{\mu}).
  \end{align*}
  By Lemma~\ref{lm:igr-eq}, the latter is zero unless $\mu=\kappa$, and is equal to $\kk$ as soon as $\mu=\kappa$.
\end{proof}

We now turn to the case when $\lambda_h > 0$. Assume that $w\geq 2$ and consider the partial flag
variety $\IFl(w-1, n; V)$ together with the two projection maps
\begin{equation}
\begin{tikzcd}
  & \IFl(w-1, n; V) \arrow[ld, "\tilde{q}" swap] \arrow[rd, "\tilde{p}"] & \\
  \IGr(w-1, V) & & \LGr(n, V).
\end{tikzcd}
\end{equation}
Denote by $\CU$ and $\CH$ the tautological bundles on $\LGr(n, V)$ and on $\IGr(w-1, V)$ respectively
as well as their pullbacks on $\IFl(w-1, n; V)$. Remark that $\PerpQ{\CH}$ is now a symplectic
vector bundle of rank $2h$. Recall that in~\eqref{eq:tw} we put for $\lambda\in\You_h$ and $t\in\ZZ$
\begin{equation*}
  \lambda(t) = (\lambda_1+t, \lambda_2+t, \ldots, \lambda_h+t).
\end{equation*}

The proof of the following proposition is very similar to that of Proposition~\ref{prop:fl-reg-sp}.
\begin{proposition}\label{prop:fl-irreg-sp}
  The exceptional object $\CF^\lambda$ is isomorphic to
  $\ptl_*\left(\det\,(\CU/\CH) \otimes \qtl^*\SpS{\left(\PerpQ{\CH}\right)}{\lambda(-1)}\right)$.
\end{proposition}
\begin{proof}
  Put $\kappa = \lambda(-1)$ and denote
  $\CF = \ptl_*\left(\det\,(\CU/\CH) \otimes \qtl^*\SpS{\left(\PerpQ{\CH}\right)}{\lambda(-1)}\right)$.
  According to Lemma~\ref{lm:sp-filtration}, there
  is a filtration on $\SpS{\left(\PerpQ{\CH}\right)}{\kappa}$ with the associated
  quotients of the form $\Sigma^\nu\!\left(\CU/\CH\right)$, where
  $-\kappa\subseteq\nu\subseteq \kappa$. Thus, the bundle
  $\det\,(\CU/\CH) \otimes \qtl^*\SpS{\left(\PerpQ{\CH}\right)}{\kappa}$
  is an iterated extension of the corresponding bundles $\Sigma^{\nu(1)}\!\left(\CU/\CH\right)$.
  As $-\kappa(1)\subseteq \nu(1)\subseteq \kappa(1)=\lambda$,
  one has $\nu(1)_h\geq -\kappa(1)_1 > -(w-1)$. From the associated
  spectral sequence and Lemma~\ref{lm:gr-0}, we see that
  $\CF \in \langle\, \Sigma^\mu\CU \mid 0\subseteq \mu \subseteq \lambda \rangle\subseteq D^b_{\bfG}(\LGr(n, V))$.

  Both projections $\ptl$ and $\qtl$ are $\bfG$-equivariant, so
  \begin{align}\label{eq:fl3}
    \Hom_\bfG^\bullet(\Sigma^\mu\CU, \CF)
    & \simeq \Hom_\bfG^\bullet(\Sigma^\mu\CU,\,\det\,(\CU/\CH) \otimes \qtl^*\SpS{(\PerpQ{\CH})}{\kappa}) \nonumber \\
    & \simeq H^\bullet(\IFl(w-1, n; V),\,\Sigma^\mu\CU^*\otimes \det\,(\CU/\CH) \otimes \qtl^*\SpS{(\PerpQ{\CH})}{\kappa})^\bfG \nonumber \\
    & \simeq H^\bullet(\IGr(w-1, V),\,\qtl_*(\Sigma^\mu\CU^*\otimes \det\,(\CU/\CH)) \otimes \SpS{(\PerpQ{\CH})}{\kappa})^\bfG.
  \end{align}
  It follows from Lemma~\ref{lm:sch-ses} that there is a filtration on
  $\Sigma^\mu\CU^*\otimes \det\,(\CU/\CH)$ with
  the associated graded of the form
  \begin{equation*}
    \bigoplus_{0\subseteq\nu\subseteq\mu} \Sigma^{\nu(-1)}\left(\CU/\CH\right)^* \otimes \Sigma^{\mu/\nu} \CH^*.
  \end{equation*}
  From Lemma~\ref{lm:lgr-0} we know that
  \begin{equation*}
    R^i\qtl_*\Sigma^{\nu(-1)}\!\left(\CU/\CH\right)^* = 
    \begin{cases}
      \SpS{(\PerpQ{\CH})}{\nu(-1)} & \quad \text{if } \nu_1 \geq 1 \text{ and } i = 0, \\
      0       & \quad \text{otherwise}.
    \end{cases}
  \end{equation*}
  Using the projection formula,
  we get a spectral sequence with the terms of the form
  \begin{equation*}
    H^\bullet(\IGr(w-1, V),\,\Sigma^{\mu/\nu} \CH^*\otimes \SpS{(\PerpQ{\CH})}{\nu(-1)} \otimes \SpS{(\PerpQ{\CH})}{\kappa})^\bfG
  \end{equation*}
  converging to~\eqref{eq:fl3}. Proceeding exactly as in the proof of Proposition~\ref{prop:fl-reg-sp},
  we conclude that the latter is zero unless $\kappa = \mu(-1)$, and is equal to $\kk$
  as soon as $\kappa = \mu(-1)$. It remains to recall that $\kappa = \lambda(-1)$.
\end{proof}

We are left with the case $w=1$, $h=n$, and $\lambda_h\geq 1$. The only such diagram
is $\lambda=(\underbrace{1,1,\ldots,1}_{n \text{ times}})$. Remark that if
$\mu=(\underbrace{1,1,\ldots,1}_{t \text{ times}}, 0, \ldots, 0)\in \You_n$, then
$\CF^\mu\simeq\Sigma^\mu\CU=\Lambda^t\CU$. Indeed, it follows from the Littlewood--Richardson
rule and Lemma~\ref{lm:lgr-0} that the objects
$\langle\CO, \CU, \ldots, \Lambda^{n-1}\CU, \Lambda^n\CU\rangle$ form a fully orthogonal
exceptional collection in $D^b_{\bfG}(\LGr(n, V))$. Thus, in for $\mu\in \You_{n, 1}$
one has $\CF^\mu\simeq \Sigma^\mu\CU$. In particular, Proposition~\ref{prop:fl-irreg-sp}
still holds once we identify $\IFl(0, n; V)$ with $\LGr(n, V)$: $\CF^\lambda\simeq \det \CU=\Lambda^n\CU$.

\begin{remark}
  \label{rm:bun}
  It follows from the proofs of Propositions~\ref{prop:fl-reg-sp}
  and~\ref{prop:fl-irreg-sp} that not only
  $\CF^\lambda\in\left\langle \Sigma^\mu\CU \mid \mu \in \You_{h,w} \right\rangle$,
  but that the objects $\CF^\lambda$ are equivariant vector bundles
  (this fact already appeared in~\cite{KP}, but it is always nice to have
  a geometric interpretation).
\end{remark}

\subsection{Second description}

Consider the isotropic Grassmannian
$\IGr(w, V)$, and recall that we denoted by $\CW$ the tautological bundle on it.
The following lemma is trivial and known; we include its proof for the sake of completeness.

\begin{lemma}\label{lm:igr-ec}
  The bundles $\left\langle\Sigma^\mu\CW^* \mid \mu \in \You_{w, h}\right\rangle$
  with any total order refining the inclusion partial order on diagrams
  form a strong (but not full) exceptional collection in $D^b(\IGr(w, V))$.
\end{lemma}
\begin{proof}
  We need to compute
  \begin{equation*}
    \Hom^\bullet(\Sigma^\mu \CW^*,\,\Sigma^\lambda \CW^*) = 
    H^\bullet(\IGr(w, V),\,\Sigma^\lambda \CW^* \otimes \Sigma^\mu \CW)
  \end{equation*}
  for a pair of diagrams $\mu, \lambda \in \You_{w, h}$.
  According to Lemma~\ref{lm:prod} the bundle $\Sigma^\lambda \CW^* \otimes \Sigma^\mu \CW$
  decomposes into a direct sum of irreducible equivariant bundles of the form
  $\Sigma^\nu\CW^*$, where $-\mu\subseteq \nu \subseteq \lambda$.
  In particular, $\nu_w \geq -\mu_1 \geq -h\geq -(2h-1)=-(2n-2w+1)$.
  It follows from Lemma~\ref{lm:igr-van} that
  \begin{equation*}
    H^i(\IGr(w, V),\, \Sigma^{\nu}\CW^*) = 
    \begin{cases}
      \SpS{V}{\nu} & \quad \text{if } \nu_w \geq 0 \text{ and } i = 0, \\
      0       & \quad \text{otherwise}.
    \end{cases}
  \end{equation*}
  Thus, $\Hom^\bullet(\Sigma^\mu \CW^*,\,\Sigma^\lambda \CW^*) = 0$ unless
  $\Sigma^\lambda \CW^* \otimes \Sigma^\mu \CW$ contains an irreducible subbundle
  $\Sigma^\nu\CU^*$ for some
  $\nu\supseteq 0$. According to Lemma~\ref{lm:pos-sum},
  the latter happens if and only if $\mu\subseteq \lambda$. If $\lambda=\mu$,
  the condition $\nu\supseteq 0$ implies $\nu=0$, and its multiplicity equals $1$;
  thus, the bundles in our collection are exceptional.
\end{proof}

Recall that $|\lambda|$ denotes the number of boxes in a Young diagram $\lambda$.

\begin{definition}\label{def:gl}
  For $\lambda \in \You_{h, w}$, define the objects $\CG^\lambda\in D^b(\IGr(w, V))$
  by the following
  property:
  \begin{equation}\label{eq:gl-def}
    \CG^\lambda \in \left\langle \Sigma^\mu \CW^* \mid \mu\in \You_{w, h}\right\rangle
    \quad\text{and}\quad
    \Hom^\bullet(\Sigma^\mu\CW^*, \CG^\lambda) =
    \begin{cases}
      \kk[-|\mu|] & \text{if } \lambda = \mu^T, \\
      0 & \text{otherwise}.
    \end{cases}
  \end{equation}
\end{definition}
A careful reader will point out that up to shifts the objects $\CG^\lambda$
coincide with the left dual exceptional
collection to $\left\langle \Sigma^\mu \CW^* \mid \mu\in \You_{w, h}\right\rangle$
(see Remark~\ref{rm:dual}).
In particular, $\CG^\lambda$ are well defined up to isomorphism.

We can finally present the promised second description of the objects $\CF^\lambda$.
The following proposition uses the notation introduced in~\eqref{eq:cd-pq}.

\begin{proposition}\label{prop:flgl}
  The object $\CF^\lambda$ is isomorphic to $p_*q^*\CG^\lambda$.
\end{proposition}

In order to prove the latter statement, we need to consider three cases,
which we treat separately: $h=1$ and $w=n$, $\lambda_h=0$, and $\lambda_h>0$,
which are treated in Propositions~\ref{prop:flgl-1}, \ref{prop:flgl-reg},
and~\ref{prop:flgl-irreg} respectively.
We begin with a simple observation that will be useful
in all these cases.

\begin{lemma}\label{lm:dual}
  Let $\nu\in\You_{n, n}$ be a Young diagram. Then the subcategories
  \begin{equation*}
    \left\langle \Sigma^\mu\CU^* \mid 0\subseteq\mu \subseteq \nu \right\rangle
    \quad \text{and} \quad
    \left\langle \Sigma^\mu\CU \mid 0\subseteq\mu \subseteq \nu^T \right\rangle
  \end{equation*}
  coincide in $D^b(\LGr(n, V))$.
\end{lemma}
\begin{proof}
  Consider the closed embedding $\iota : \LGr(n, V) \to \Gr(n, V)$. Remark that the tautological
  bundle on $\LGr(n, V)$ is the restriction of the tautological bundle on $\Gr(n, V)$;
  we denote both by $\CU$. Moreover, the Lagrangian condition implies
  $\iota^*(V/\CU)^* \simeq \CU$. Kapranov showed in~\cite{Kap} that for any $\nu\in \You_{n, n}$
  the~bundles $\left\langle \Sigma^\mu\CU^* \mid 0\subseteq\mu \subseteq \nu \right\rangle$ form
  an exceptional collection in $D^b(\Gr(n, V))$, and the left dual to this collection
  is $\left\langle \Sigma^\mu(V/\CU)^* \mid 0\subseteq\mu \subseteq \nu^T \right\rangle$.
  Once we apply $\iota^*$, the claim follows immediately.
\end{proof}

\begin{proposition}\label{prop:dual}
  The object $p_*q^*\CG^\lambda$ belongs to the subcategory
  \begin{equation*}
    \left\langle \Sigma^\mu\CU \mid \mu \in \You_{h, w}\right\rangle \subset D^b(\LGr(n, V)).
  \end{equation*}
\end{proposition}
\begin{proof}
  By definition, the object $\CG^\lambda$ belongs to the subcategory
  \begin{equation*}
    \left\langle \Sigma^\mu\CW^* \mid \mu \in \You_{w, h}\right\rangle \subset D^b(\IGr(w, V)).
  \end{equation*}
  It follows from Lemma~\ref{lm:gr-0} that $p_*q^*\Sigma^\mu\CW^* = \Sigma^\mu\CU^*$
  for any $\mu \in \YD_w$.
  Thus,
  \begin{equation*}
    p_*q^*\CG^\lambda \in \left\langle \Sigma^\mu\CU^* \mid \mu \in \You_{w, h}\right\rangle \subset D^b(\LGr(n, V)).
  \end{equation*}
  The claim now follows from Lemma~\ref{lm:dual} applied to $\nu = (\underbrace{h, h, \ldots, h}_{w\text{ times}})$.
\end{proof}

\begin{proposition}\label{prop:flgl-1}
  Proposition~\ref{prop:flgl} holds when $h=1$ and $w=n$.
\end{proposition}
\begin{proof}
  We need to prove that the bundles $\CF^{(k)}$ form a left dual exceptional collection
  to the collection $\left\langle \CO, \CU^*, \Lambda^2\CU^*, \ldots, \Lambda^n\CU^* \right\rangle$
  in the sense of Definition~\ref{def:gl}. From Proposition~\ref{prop:dual} we know
  that the object $\CF^{(k)}$ belongs to the subcategory
  $\left\langle \CO, \CU^*, \Lambda^2\CU^*, \ldots, \Lambda^n\CU^* \right\rangle$.
  It remains to show that
  \begin{equation}\label{eq:flgl-1}
    \Hom^\bullet(\Lambda^i\CU^*, \CF^{(j)}) =
    \begin{cases}
      \kk[-i] & \text{if } i = j, \\
      0 & \text{otherwise}.
    \end{cases}
  \end{equation}
  If $j=0$, $\CG^{(0)} = \CO$, and the statement follows from exceptionality of
  $\left\langle \CO, \CU^*, \Lambda^2\CU^*, \ldots, \Lambda^n\CU^* \right\rangle$. If $j>0$,
  consider the diagram
  \begin{equation*}
    \begin{tikzcd}
      & \IFl(n-1, n; V) \arrow[ld, "\tilde{q}" swap] \arrow[rd, "\tilde{p}"] & \\
      \IGr(n-1, V) & & \LGr(n, V).
    \end{tikzcd}
  \end{equation*}
  Let $\CH\subset\CU$ denote the universal flag on $\IFl(n-1,n;V)$.
  By Proposition~\ref{prop:fl-irreg-sp},
  \begin{equation*}
    \CF^{(j)}\simeq \ptl_*\left(\det\,(\CU/\CH) \otimes \qtl^*\SpS{(\PerpQ{\CH})}{j-1}\right) =
    \ptl_*\left((\CU/\CH) \otimes \qtl^*S^{(j-1)}(\PerpQ{\CH})\right),
  \end{equation*}
  where $\SpS{(\PerpQ{\CH})}{j-1}\simeq S^{(j-1)}(\PerpQ{\CH})$ since $\PerpQ{\CH}$ is $2$-dimensional.
  Now,
  \begin{align}
    \Hom^\bullet(\Lambda^i\CU^*, \CG^{(j)})
    & \simeq \Hom^\bullet(\Lambda^i\CU^*, \ptl_*((\CU/\CH) \otimes \qtl^*S^{(j-1)}(\PerpQ{\CH}))) \nonumber \\
    & \simeq \Hom^\bullet(\Lambda^i\CU^*, (\CU/\CH) \otimes \qtl^*S^{(j-1)}(\PerpQ{\CH})) \nonumber \\
    & \simeq H^\bullet(\IFl(n-1, n; V),\,\Lambda^i\CU\otimes (\CU/\CH) \otimes \qtl^*S^{(j-1)}(\PerpQ{\CH})) \nonumber \\
    & \simeq H^\bullet(\IGr(n-1, V),\,\qtl_*\left(\Lambda^i\CU\otimes (\CU/\CH)\right) \otimes S^{(j-1)}(\PerpQ{\CH})). \nonumber
  \end{align}
  
  If $i = 0$, then $\qtl_*\left(\Lambda^i\CU\otimes (\CU/\CH)\right) = \qtl_*(\CU/\CH) = 0$ as
  $\IFl(w-1, w; V) \simeq \PP(\PerpQ{\CH})$ is the projectivization of the rank~$2$ bundle
  $\PerpQ{\CH}$ and $\CU/\CH$ is the relative tautological
  line bundle. Thus, \eqref{eq:flgl-1} holds for $i = 0$.

  If $i>0$, one has a short exact sequence
  $0\to \Lambda^i\CH \to \Lambda^i\CU\to \Lambda^{i-1}\CH\otimes (\CU/\CH)\to 0$. Twisting
  it by $\CU/\CH$, we obtain a short exact sequence
  \begin{equation}\label{eq:flgl-1-ses}
    0\to \Lambda^i\CH\otimes (\CU/\CH) \to \Lambda^i\CU\otimes (\CU/\CH)
    \to \Lambda^{i-1}\CH\otimes (\CU/\CH)^{\otimes 2}\to 0.
  \end{equation}
  As $\qtl_*(\Lambda^i\CH\otimes (\CU/\CH))\simeq \Lambda^i\CH\otimes \qtl_*(\CU/\CH) \simeq 0$,
  we conclude that
  \begin{equation*}
    \qtl_*(\Lambda^i\CU\otimes (\CU/\CH)) \simeq
    % \qtl_*(\Lambda^{i-1}\CH\otimes (\CU/\CH)^{\otimes 2}) \simeq
    \Lambda^{i-1}\CH\otimes \qtl_*(\CU/\CH)^{\otimes 2} \simeq
    \Lambda^{i-1}\CH\otimes \det\,\PerpQ{\CH}[-1] \simeq
    \Lambda^{i-1}\CH[-1].
  \end{equation*}
  Thus,
  \begin{align}
    \Hom^\bullet(\Lambda^i\CU^*, \CG^{(j)})
    & = H^\bullet(\IGr(n-1, V),\,\qtl_*\left(\Lambda^i\CU\otimes (\CU/\CH)\right) \otimes S^{(j-1)}(\PerpQ{\CH})) \nonumber \\
    & = H^\bullet(\IGr(n-1, V),\,\Lambda^{i-1}\CH \otimes S^{j-1}(\PerpQ{\CH}))[-1]. \nonumber 
  \end{align}
  From Lemma~\ref{lm:igr-kap} we know that
  \begin{equation*}
    H^\bullet(\IGr(n-1, V),\,\Lambda^{i-1}\CH \otimes S^{(j-1)}(\PerpQ{\CH})) =
    \begin{cases}
      \kk[-i+1] & \text{if } i = j, \\
      0 & \text{otherwise},
    \end{cases}
  \end{equation*}
  which finishes the proof.
\end{proof}

Until the end of this section we assume that $w<n$.
By Lemma~\ref{lm:igr-ec}, the collection $\left\langle \Sigma^\mu\CW^* \mid \mu\in\You_{w, h}\right\rangle$
is exceptional in $D^b(\IGr(w, V))$. Thus, it generates an admissible full triangulated
subcategory, which we denote by $\CA$.
Consider the semiorthogonal decomposition
\begin{equation}\label{eq:ca}
  D^b(\IGr(w, V)) = \langle \CA^\perp, \CA\rangle.
\end{equation}
With any object $Z\in D^b(\IGr(w, V))$ one can associate a functorial triangle
\begin{equation*}
  X \to Z\to Y\to X[1],
\end{equation*}
where $X\in \CA$ and $Y\in \CA^\perp$ are the projections of $Z$ on $\CA$
and $\CA^\perp$ respectively.

\begin{lemma}\label{lm:perp0}
  Let $Y\in \CA^\perp$, and let $\mu\in\You_{n, h}$.
  Then $\Hom^\bullet(\Sigma^\mu\CU^*,\, p_*q^*Y)=0$.
\end{lemma}
\begin{proof}
  We need to show the vanishing of
  \begin{align*}
    \Hom^\bullet(\Sigma^\mu\CU^*,\, p_*q^*Y)
    & \simeq \Hom^\bullet(\Sigma^\mu\CU^*,\, q^*Y) \\
    & \simeq H^\bullet(\IFl(w, n; V),\, \Sigma^\mu\CU\otimes q^*Y) \\
    & \simeq H^\bullet(\IGr(w, V),\, q_*(\Sigma^\mu\CU)\otimes Y).
  \end{align*}
  By Lemma~\ref{lm:sch-ses}, there is a filtration on $\Sigma^\mu\CU^\mu$ with the
  associated quotients of the form $\Sigma^\nu(\CU/\CW)\otimes \Sigma^{\mu/\nu}\CW$,
  where $0\subseteq \nu \subseteq \mu$, while $\Sigma^{\mu/\nu}\CW$ splits into a direct
  sum of equivariant vector bundles of the form $\Sigma^\tau\CW$ with $0\subseteq \tau\subseteq \mu$
  (in particular, $\Sigma^\tau\CW=0$ if $\tau$ has more than $w$ rows, and
  $\Sigma^\tau\CW\neq 0$ otherwise).
  Looking at the associated spectral sequence, we see that it is enough to show that
  for any pair of diagrams $\alpha,\beta\subseteq \mu$ one has
  \begin{equation}\label{eq:lm-perp0}
    H^\bullet(\IGr(w, V),\, q_*\Sigma^\alpha(\CU/\CW)\otimes \Sigma^\beta \CW \otimes Y) = 0.
  \end{equation}
  By our assumptions, $\alpha_1\leq h$. Thus, by Lemma~\ref{lm:q-neg}, either
  $q_*\Sigma^\alpha(\CU/\CW) = 0$, or $q_*\Sigma^\alpha(\CU/\CW)\simeq \CO[t]$ for some $t\in\ZZ$.
  In the first case, the cohomology groups~\eqref{eq:lm-perp0} vanish, while in the
  second case,
  \begin{equation*}
    H^\bullet(\IGr(w, V),\, q_*\Sigma^\alpha(\CU/\CW)\otimes \Sigma^\beta \CW \otimes Y)
    \simeq
    H^\bullet(\IGr(w, V),\, \Sigma^\beta \CW \otimes Y[t])
    \simeq
    \Hom(\Sigma^\beta\CW^*,\, Y[t]) = 0,
  \end{equation*}
  since $Y\in \CA^\perp$, and $\Sigma^\beta\CW^*\in\CA$ (when $\beta$ has more than $w$
  rows, $\Sigma^\beta\CW^*=0$).
\end{proof}

Put $\CB = \left\langle \Sigma^\mu\CU \mid \mu \in \You_{h, w}\right\rangle$.

\begin{lemma}\label{lm:perp}
  Let $Y\in \CA^\perp$. Then $p_*q^*Y\in \CB^\perp$.
\end{lemma}
\begin{proof}
  Remark that by Lemma~\ref{lm:dual},
  \begin{equation*}
    \CB = \left\langle \Sigma^\mu\CU^* \mid \mu \in \You_{w, h}\right\rangle
  \end{equation*}
  As $\You_{w,h}\subseteq\You_{n, h}$, the statement follows immediately from
  Lemma~\ref{lm:perp0}.
\end{proof}

\begin{lemma}\label{lm:two-three}
  Let $Z\in D^b(\IGr(w, V))$ be such that $p_*q^*Z\in \CB$, and let
  $X$ be the projection of $Z$ on $\CA$ with respect to the semiorthogonal
  decomposition~\eqref{eq:ca}. Then $p_*q^*X\simeq p_*q^*Z$.
\end{lemma}
\begin{proof}
  Consider the exact triangle
  \(
    X \to Z\to Y\to X[1],
  \)
  where $X\in \CA$ and $Y\in \CA^\perp$ are the projections of $Z$ on $\CA$
  an $\CA^\perp$ respectively.
  Once we apply the functor $p_*q^*$ to it, we obtain a triangle of the form
  \begin{equation}\label{eq:tr-two-three}
    p_*q^*X \to p_*q^*Z\to p_*q^*Y\to p_*q^*X[1].
  \end{equation}
  It follows from Proposition~\ref{prop:dual} that $p_*q^*X\in \CB$, while
  $p_*q^*Z\in \CB$ by our assumptions; thus, $p_*q^*Y\in \CB$.
  Meanwhile, by Lemma~\ref{lm:perp}, $p_*q^*Y\in \CB^\perp$. We conclude
  that $p_*q^*Y\simeq 0$, which implies that the first morphism in
  the triangle~\eqref{eq:tr-two-three} is an isomorphism.
\end{proof}

\begin{proposition}\label{prop:flgl-reg}
  Proposition~\ref{prop:flgl} holds when $\lambda_h = 0$.
\end{proposition}
\begin{proof}
  By Proposition~\ref{prop:fl-reg-sp} and Lemma~\ref{lm:two-three},
  it is sufficient to show that the projection of $\SpS{(\PerpQ{\CW})}{\lambda}$
  on $\CA$ with respect to the decomposition~\eqref{eq:ca} is isomorphic to $\CG^\lambda$.

  Consider the exact triangle
  \begin{equation*}
    X \to \SpS{(\PerpQ{\CW})}{\lambda} \to Y \to X[1]
  \end{equation*}
  associated with the semiorthogonal decomposition~\eqref{eq:ca}.
  For any $\mu \in \You_{w, h}$ one has
  \begin{equation}
    \Hom^\bullet(\Sigma^\mu\CU^*,\, X) \simeq \Hom^\bullet(\Sigma^\mu\CU^*,\, \SpS{(\PerpQ{\CW})}{\lambda}).
  \end{equation}
  Lemma~\ref{lm:igr-kap} shows that
  \begin{equation*}
    \Hom^\bullet(\Sigma^\mu\CU^*,\, \SpS{(\PerpQ{\CW})}{\lambda}) =
    \begin{cases}
      \kk[-|\mu|] & \text{if }\mu = \lambda^T, \\
      0 & \text{otherwise}.
    \end{cases}
  \end{equation*}
  Since $X\in \CA$, the latter implies that $X$ satisfies
  the defining conditions~\eqref{eq:gl-def} of $\CG^\lambda$.
\end{proof}

We now turn to the harder case $\lambda_h > 0$. Consider joint following diagram.

\begin{equation}\label{eq:d-big}
  \begin{tikzcd}[column sep=tiny]
    & & \IFl(w-1, w, n; V) \arrow[ld, "r" swap] \arrow[rd, "s"] & & \\
    & \IFl(w-1, n; V) \arrow[ld, "\tilde{q}" swap] \arrow[rd, "\tilde{p}"] & & \IFl(w, n; V) \arrow[ld, "p" swap] \arrow[rd, "q"] & \\
    \IGr(w-1, V) & & \LGr(n, V) & & \IGr(w, V).
  \end{tikzcd}
\end{equation}
Recall that the universal flag on $\IFl(w-1, w, n; V)$ is denoted by $\CH\subseteq\CW\subseteq\CU$.

\begin{lemma}\label{lm:flgl-irreg}
  Let $\lambda\in\You_{h, w}$ be a Young diagram with $\lambda_h > 0$.
  Then 
  \begin{equation*}
    \CF^\lambda\simeq \ptl_*r_*\left((\CW/\CH)^{\otimes h} \otimes r^*\qtl^*\SpS{\left(\PerpQ{\CH}\right)}{\lambda(-1)}\right)[h].
  \end{equation*}
\end{lemma}
\begin{proof}
  Recall that $\IFl(w-1,w,n;V)$ is naturally isomorphic to the projectivization
  $\PP_{\IFl(w-1,n;V)}(\CU/\CH)$. Under this identification $r$ is nothing but the
  projection morphism, while $(\CW/\CH)^*$ is the relative very ample line bundle.
  Now, using the projection formula we see that
  \begin{align*}
    \ptl_*r_*\left((\CW/\CH)^{\otimes h} \otimes r^*\qtl^*\SpS{(\PerpQ{\CH})}{\lambda(-1)}\right)[h]
    & \simeq
    \ptl_*\left(r_*(\CW/\CH)^{\otimes h} \otimes \qtl^*\SpS{(\PerpQ{\CH})}{\lambda(-1)}\right)[h] \\
    & \simeq
    \ptl_*\left(\det \left(\CU/\CH\right)[-h] \otimes \qtl^*\SpS{(\PerpQ{\CH})}{\lambda(-1)}\right)[h] \\
    & \simeq
    \ptl_*\left(\det \left(\CU/\CH\right) \otimes \qtl^*\SpS{(\PerpQ{\CH})}{\lambda(-1)}\right)
    \simeq \CF^\lambda,
  \end{align*}
  where the last isomorphism was established in Proposition~\ref{prop:fl-irreg-sp}.
\end{proof}

\begin{proposition}\label{prop:flgl-irreg}
  Proposition~\ref{prop:flgl} holds when $\lambda_h > 0$.
\end{proposition}
\begin{proof}
  In the previous lemma we established that
  \begin{equation*}
    \CF^\lambda\simeq \ptl_*r_*\left((\CW/\CH)^{\otimes h} \otimes r^*\qtl^*\SpS{(\PerpQ{\CH})}{\lambda(-1)}\right)[h].
  \end{equation*}
  Using commutativity of the diagram~\eqref{eq:d-big}, we can rewrite
  \begin{equation*}
    \ptl_*r_*\left((\CW/\CH)^{\otimes h} \otimes r^*\qtl^*\SpS{(\PerpQ{\CH})}{\lambda(-1)}\right)[h]
    \simeq
    p_*s_*\left((\CW/\CH)^{\otimes h} \otimes r^*\qtl^*\SpS{(\PerpQ{\CH})}{\lambda(-1)}\right)[h].
  \end{equation*}
  Consider the diagram
  \begin{equation}\label{eq:d-med}
    \begin{tikzcd}[column sep=tiny]
      & & \IFl(w-1, w, n; V) \arrow[ld, "s" swap] \arrow[rd, "q'"] & & \\
      & \IFl(w, n; V) \arrow[ld, "p" swap] \arrow[rd, "q"] & & \IFl(w-1, w; V) \arrow[ld, "s'" swap] \arrow[rd, "p'"] & \\
       \LGr(n, V) & & \IGr(w, V) & & \IGr(w-1, V).
    \end{tikzcd}
  \end{equation}
  Remark that $p'q'=\qtl r$, and that the line bundle $(\CW/\CH)$ on $\IFl(w-1,w,n;V)$ is
  pulled back from $\IFl(w-1,w;V)$. Thus,
  \begin{equation*}
    p_*s_*\left((\CW/\CH)^{\otimes h} \otimes r^*\qtl^*\SpS{(\PerpQ{\CH})}{\lambda(-1)}\right)[h]
    \simeq
    p_*s_*q'^*\left((\CW/\CH)^{\otimes h} \otimes p'^*\SpS{(\PerpQ{\CH})}{\lambda(-1)}\right)[h].
  \end{equation*}
  Since the middle square in~\eqref{eq:d-med} is Cartesian and $\Tor$-independent, we conclude
  that
  \begin{equation*}
    \CF^\lambda \simeq p_*q^*s'_*\left((\CW/\CH)^{\otimes h} \otimes p'^*\SpS{(\PerpQ{\CH})}{\lambda(-1)}\right)[h].
  \end{equation*}
  For convenience, we put $\CF = s'_*\left((\CW/\CH)^{\otimes h} \otimes p'^*\SpS{(\PerpQ{\CH})}{\lambda(-1)}\right)[h]$,
  then the previous equation reads
  \begin{equation}\label{eq:flgl-irreg}
    \CF^\lambda \simeq p_*q^*\CF.
  \end{equation}

  By Lemma~\ref{lm:two-three}, it is enough to show that the projection of $\CF$
  on $\CA$ with respect to the decomposition~\eqref{eq:ca} is isomorphic to $\CG^\lambda$.
  In order to do that, for all $\mu\in\You_{w, h}$ we compute
  \begin{align*}
    \Hom^\bullet(\Sigma^\mu\CW^*,\, \CF)
    & \simeq \Hom^\bullet(\Sigma^\mu\CW^*,\, s'_*((\CW/\CH)^{\otimes h} \otimes p'^*\SpS{(\PerpQ{\CH})}{\lambda(-1)})[h]) \\
    & \simeq \Hom^\bullet(\Sigma^\mu\CW^*,\, (\CW/\CH)^{\otimes h} \otimes p'^*\SpS{(\PerpQ{\CH})}{\lambda(-1)}[h]) \\
    & \simeq H^\bullet(\IFl(w-1, w; V),\, \Sigma^\mu\CW\otimes (\CW/\CH)^{\otimes h} \otimes p'^*\SpS{(\PerpQ{\CH})}{\lambda(-1)})[h] \\
    & \simeq H^\bullet(\IGr(w-1, V),\, p'_*(\Sigma^\mu\CW\otimes (\CW/\CH)^{\otimes h}) \otimes \SpS{(\PerpQ{\CH})}{\lambda(-1)})[h]. \\
  \end{align*}

  By Lemma~\ref{lm:sch-ses}, the bundle $\Sigma^\mu\CW$ is an iterated extension of the
  bundles $(\CW/\CH)^{\otimes i}\otimes \Sigma^{\mu/(i)}\CH$, where $i=0,\ldots,\mu_1$.
  It remains to compute
  \begin{equation}\label{eq:flgl-irreg-ss}
    H^\bullet(\IGr(w-1, V),\, p'_*(\CW/\CH)^{\otimes (i+h)}\otimes \Sigma^{\mu/(i)}\CH \otimes \SpS{(\PerpQ{\CH})}{\lambda(-1)})[h].
  \end{equation}

  Since $\IGr(w-1,w;V)$ is isomorphic to the projectivization of the rank $2h$
  vector bundle $\PerpQ{\CW}$, and
  $(\CW/\CH)$ is isomorphic to the tautological line bundle, $p'_*(\CW/\CH)^{\otimes (i+h)}=0$
  for $0\leq i < h$, and $p'_*(\CW/\CH)^{\otimes (i+h)}\simeq \det\PerpQ{\CH}[-2h]\simeq\CO[-2h]$ when $i=h$.
  We conclude that~\eqref{eq:flgl-irreg-ss} vanishes when $\mu_1<h$,
  while when $\mu_1=h$, the only potentially non-zero cohomology (corresponding to $i=h$) is
  \begin{multline*}
    H^\bullet(\IGr(w-1, V),\, p'_*(\CW/\CH)^{\otimes 2h}\otimes \Sigma^{\mu/(h)}\CH \otimes \SpS{(\PerpQ{\CH})}{\lambda(-1)})[h] \\
    \simeq H^\bullet(\IGr(w-1, V),\, \Sigma^{\bar{\mu}}\CH \otimes \SpS{(\PerpQ{\CH})}{\lambda(-1)})[-h],
  \end{multline*}
  where $\bar{\mu} = (\mu_2,\mu_3,\ldots,\mu_{w-1})\in\You_{w-1,h}$. Lemma~\ref{lm:igr-kap} applied
  to $\IGr(w-1, V)$ implies that
  \begin{equation*}
    H^\bullet(\IGr(w-1, V),\, \Sigma^{\bar{\mu}}\CH \otimes \SpS{(\PerpQ{\CH})}{\lambda(-1)})[-h] =
    \begin{cases}
      \kk[-|\lambda(-1)|-h] & \text{if }\lambda(-1)=\bar{\mu}^T, \\
      0 & \text{otherwise}.
    \end{cases}
  \end{equation*}

  Since $\bar{\mu}^T=\mu^T(-1)$ and $|\lambda(-1)|+h = |\lambda|$,
  we conclude that the projection of $\CF$ on $\CA$ satisfies the~universal property~\eqref{eq:gl-def} of $\CG^\lambda$.
\end{proof}

\begin{proof}[Proof of Proposition~\ref{prop:flgl}]
  Combine the proofs of Propositions~\ref{prop:flgl-1}, \ref{prop:flgl-reg},
  and~\ref{prop:flgl-irreg}.
\end{proof}

\section{Staircase complexes}

\subsection{Combinatorial setup}\label{ssec:comb}
Let $a_0a_1\ldots a_n\in\{0,1\}^{n+1}$ be a binary sequence of length $n$. We consider the
operation
\begin{equation*}
  a_0a_1\ldots a_n \mapsto (1-a_n)a_0a_1\ldots a_{n-1},
\end{equation*}
which in an obvious way defines an action of the cyclic group $G=\ZZ/(2n+2)$ on $\{0,1\}^{n+1}$.
The set of such binary sequences is in bijection with $\bigsqcup \You_{h, w}$,
where $h=0,\ldots, n+1$, and $h+w=n+1$: a given sequence defines a integral path
from the lower left to the upper right corner of a rectangle of hight $h$ and width $w$,
where $h$ is the number of times $0$ appears in the sequence.
The induced action of $G$ on $\bigsqcup \You_{h, w}$ is slightly less pleasant to describe:
the generator sends $\lambda\in\You_{h, w}$ to $\lambda'$, where
\begin{equation}\label{eq:lprime}
  \lambda' = \begin{cases}
    (\lambda_1, \lambda_2,\ldots,\lambda_h, 0) & \text{if } \lambda_1 < w, \\
    (\lambda_2+1, \lambda_3+1, \ldots, \lambda_h+1) & \text{if } \lambda_1 = w.
  \end{cases}
\end{equation}
In particular, $\lambda' \in \You_{h+1,w-1}$ in the first case, and
$\lambda' \in \You_{h-1, w+1}$ in the second case.

Let us now fix a pair of integers $w,h>0$ such that $w+h=n+1$.
Given a diagram $\lambda\in\You_{h, w}$ with $\lambda_1=w$, we define a sequence of
diagrams $\lambda^{(1)},\lambda^{(2)}, \ldots, \lambda^{(w)}$ by the following rule.
For $0<i\leq w$, let $j$ be the largest index such that $\lambda_j>w-i$.
Then
\begin{equation*}
  \lambda^{(i)} = (\lambda_2-1, \lambda_3-1,\ldots, \lambda_j-1, w-i, \lambda_{j+1},\ldots, \lambda_h).
\end{equation*}
Remark that $\lambda^{(i)}\subset \lambda$, and put $\nu_i = |\lambda/\lambda^{(i)}|$.

\begin{example}
  Let $n=5$, $h=3$, and $w=3$. Let $\lambda=(3, 3, 1)$. Then
  \begin{equation}
    \lambda^{(1)} = (2, 2, 1),\quad \lambda^{(2)} = (2, 1, 1), \quad \text{and} \quad
    \lambda^{(3)} = (2, 0, 0).
  \end{equation}
\end{example}

\begin{remark}
  The diagrams $\lambda^{(i)}$ already appeared in~\cite{Fon-LD}, see Remark~\ref{rm:st-gr}.
\end{remark}

\subsection{Staircase complexes}

The following proposition will be our main tool in the proof of fullness of the
Kuznetsov--Polishchuk exceptional collection.

Given a $2n$-dimensional symplectic vector space $V$ and an integer $1\leq i\leq n$,
we denote by $V^{[i]}$ the~$i$-th fundamental representation of the group $\SP(V)$.
In other words, $V^{[i]}=\SpS{V}{(i)^T}$.

\begin{proposition}\label{prop:stc}
  Let $w$ and $h$ be positive integers such that $w+h=n+1$. Let $\lambda\in\You_{h, w}$
  be such that $\lambda_1=w$.
  There is an exact complex of vector bundles on $\LGr(n, V)$ of the form
  \begin{equation}\label{eq:stc}
    0\to \CE^{\lambda'}(-1)\to V^{[\nu_w]}\otimes\CE^{\lambda^{(w)}}\to \cdots
    \to V^{[\nu_2]}\otimes\CE^{\lambda^{(2)}}\to V^{[\nu_1]}\otimes\CE^{\lambda^{(1)}}\to
    \CE^\lambda\to 0.
  \end{equation}
\end{proposition}

We call the complexes of the form~\ref{eq:stc} \emph{Lagrangian staircase complexes}.

\begin{remark}\label{rm:st-gr}
  Let $V$ be an $n$-dimensional vector space, and let $0 < k < n$ be an integer.
  For any diagram $\lambda\in\You_{k, n-k}$ with $\lambda_1=n-k$ there is an exact
  complex of vector bundles on $\Gr(k, V)$ of the form
  \begin{equation}\label{eq:stc-gr}
    0\to \Sigma^{\bar{\lambda}}\CU^*(-1)\to V^{\nu_w}\otimes\Sigma^{\lambda^{(w)}}\CU^*\to \cdots
    \to V^{\nu_2}\otimes\Sigma^{\lambda^{(2)}}\CU^*\to V^{\nu_1}\otimes\Sigma^{\lambda^{(1)}}\CU^*\to
    \Sigma^\lambda\CU^*\to 0,
  \end{equation}
  where $\nu_i$ and $\lambda^{(i)}$ are the same as above, $\bar{\lambda}=(\lambda_2,\ldots,\lambda_k,0)$,
  and~$V^{i}=\Lambda^i V^*$ is the $i$-th fundamental representation of the group $\GL(V)$.
  Complexes of the form~\eqref{eq:stc-gr} are called \emph{staircase}, see~\cite{Fon-LD}
  for details.
\end{remark}

Let $w$ and $h$ be positive integers such that $w+h=n+1$. We further assume that $w<n$.
Consider the subcategory
\begin{equation*}
  \CB' = \left\langle \Sigma^\mu\CU \mid \mu \in \You_{h-1, w+1}\right\rangle \subseteq \LGr(n, V).
\end{equation*}
Consider the diagram~\eqref{eq:cd-pq}.
Recall that $\CA\subseteq \IGr(w, V)$ was defined in~\eqref{eq:ca}.
The following lemma is very similar to Lemma~\ref{lm:perp}.

\begin{lemma}\label{lm:perp-pr}
  Let $Y\in \CA^\perp$. Then $p_*q^*Y\in \CB'(1)^\perp$.
\end{lemma}
\begin{proof}
  Remark that by Lemma~\ref{lm:dual},
  \begin{equation*}
    \CB'(1) = \left\langle \Sigma^\mu\CU^*(1) \mid \mu \in \You_{w+1, h-1}\right\rangle
  \end{equation*}
  Since $\Sigma^\mu\CU^*(1)\simeq \Sigma^{\mu(1)}\CU^*$, and $\mu_1\leq h-1$,
  the statement follows immediately from Lemma~\ref{lm:perp0}.
\end{proof}

\begin{proof}[Proof of Proposition~\ref{prop:stc}]
  Instead of constructing the complex~\eqref{eq:stc}, we will construct its dual
  \begin{equation}\label{eq:stc-dual}
    0\to \CF^\lambda\to V^{[\nu_1]}\otimes\CF^{\lambda^{(1)}}\to
    V^{[\nu_2]}\otimes\CF^{\lambda^{(2)}}\to \cdots \to
    V^{[\nu_w]}\otimes\CF^{\lambda^{(w)}}\to \CF^{\lambda'}(1) \to 0.
  \end{equation}

  Let us first consider the case when $w<n$.
  By definition, $\lambda'=(\lambda_2+1, \lambda_3+1, \ldots, \lambda_h+1)\in\You_{h-1,w+1}$.
  Thus, by Proposition~\ref{prop:fl-irreg-sp},
  \begin{equation*}
  \CF^{\lambda'}\simeq p_*\left(\det\,(\CU/\CW) \otimes q^*\SpS{(\PerpQ{\CW})}{\lambda'(-1)}\right)
  \end{equation*}
  (remark that $p$ and $q$ are as in~\eqref{eq:cd-pq} since the height of $\lambda'$ is $h-1$).
  Using the projection formula together with the isomorphism
  $p^*\CO(1)\simeq \det\,\CU^*\simeq \det\CW^*\otimes \det\,(\CU/\CW)^*$, we conclude that
  \begin{equation}
    \CF^{\lambda'}(1)\simeq p_*q^*\left(\det\CW^* \otimes \SpS{(\PerpQ{\CW})}{\bar{\lambda}}\right),
  \end{equation}
  where $\bar{\lambda}=(\lambda_2,\lambda_3,\ldots,\lambda_h)$.

  Put $\CF = \det\CW^* \otimes \SpS{(\PerpQ{\CW})}{\bar{\lambda}}$. Consider the exact
  triangle in $D^b(\IGr(w, V))$ induced by the semiorthogonal
  decomposition~\eqref{eq:ca}:
  \begin{equation}\label{eq:stc-tr}
    X \to \CF\to Y\to X[1].
  \end{equation}
  Recall that $X\in\CA$ and $Y\in \CA^\perp$. Let us apply the functor $p_*q^*$ to
  the triangle~\eqref{eq:stc-tr}. We get an exact triangle in $D^b(\LGr(n, V))$ of
  the form
  \begin{equation}\label{eq:stc-trpq}
    p_*q^*X \to \CF^{\lambda'}(1) \to p_*q^*Y\to p_*q^*X[1].
  \end{equation}
  First of all, we claim that $p_*q^*Y=0$. Indeed, by Proposition~\ref{prop:dual}
  the object $p_*q^*X$ belongs to the~subcategory $\CB=\langle \Sigma^\mu\CU \mid \mu\in\You_{h,w}$,
  while $\CF^{\lambda'}(1)$ by definition belongs to $\CB'(1)$.
  By Lemmas~\ref{lm:perp} and~\ref{lm:perp-pr}, $p_*q^*Y\in \CB^\perp \cap \CB'(1)^\perp$.
  Thus, $p_*q^*Y=0$.

  Recall that $\CA$ is generated by the exceptional collection
  $\langle \CG^\lambda \mid \lambda\in \You_{h,w}\rangle$, which is left dual to the~collection
  $\langle \Sigma^\mu\CW^* \mid \mu\in \You_{w,h}\rangle$ in the sense of Definition~\ref{def:gl}.
  Fix a total ordering on $\You_{h,w}$ compatible with the reverse inclusion partial order
  on $\You_{h,w}$. There is a filtration on $X$ of the form
  \begin{equation}\label{eq:stc-fil}
    0=X_0\to\cdots\to X_{\lambda_{i-1}}\to X_{\lambda_{i}}\to \cdots \to X_{(w,w,\ldots,w)}=X,
  \end{equation}
  such that the cone $C_{\lambda_i}=\Cone(X_{\lambda_{i-1}}\to X_{\lambda_{i}})$ belongs to the subcategory
  $\langle \CG^{\lambda_i}\rangle$. In order to compute $C_{\lambda}$, it is enough
  to compute
  \begin{equation*}
    \Ext^\bullet(\Sigma^\mu\CW^*, X) \simeq \Ext^\bullet(\Sigma^\mu\CW^*, \CF)\simeq
    H^\bullet(\IGr(k, V), \Sigma^{\mu(-1)}\CW\otimes \SpS{(\PerpQ{\CW})}{\bar{\lambda}}),
  \end{equation*}
  which is done in Proposition~\ref{prop:main}.

  Once we through away the repeating terms of the filtration, one gets a filtration
  \begin{equation*}
    0=Z_{w+1}\to Z_w\to \cdots \to Z_1\to X_0 = X,
  \end{equation*}
  where $\Cone(Z_{i+1}\to Z_i) = V^{[\nu_i]}\otimes\CG^{\lambda^{(i)}}[w-i]$.
  It follows from Proposition~\ref{prop:flgl} that
  there is a filtration on $p_*q^*X\simeq \CF^{\lambda'}(1)$ with the associated
  quotients of the form $V^{[\nu_i]}\otimes\CF^{\lambda^{(i)}}[w-i]$. Since all $\CF^\mu$ are sheaves,
  the associated spectral sequence degenerates into an exact complex of the form~\eqref{eq:stc-dual}.
  
  We are left with the case $w=n$. Precisely, we need to construct
  an exact complex of the form
  \begin{equation}\label{eq:stc-wn}
    0\to \CF^n\to V\otimes \CF^{n-1}\to V^{[2]}\otimes\CF^{n-2}\to \cdots \to
    V^{[n]}\otimes\CO\to \CO(1)\to 0.
  \end{equation}
  By Proposition~\ref{prop:flgl-1}, we know that the objects
  $\langle \CF^n, \CF^{n-1},\ldots,\CF^1, \CO \rangle$ form a left dual
  exceptional collection to $\langle \CO, \CU^*,\ldots, \Lambda^{n-1}\CU^*, \Lambda^n\CU^*\rangle$
  in the sense of definition~\ref{def:gl}. Since $\Lambda^n\CU^*\simeq \CO(1)$, we conclude
  that $\CO(1)\in \langle \CF^n, \CF^{n-1},\ldots,\CF^1, \CO \rangle$.
  Thus, there is a filtration in $D^b(\LGr(n, V))$ of the form
  \begin{equation}\label{eq:fil-wn}
    0=X_0\to X_1\to\cdots\to X_n\to X_{n+1}=\CO(1)
  \end{equation}
  such that the cone $\Cone(X_{i}\to X_{i+1})$ belongs to the subcategory $\langle \CF^i\rangle$.

  As in the general case, we compute
  \begin{equation*}
    \Ext^\bullet(\Lambda^i\CU^*, \CO(1)) \simeq H^\bullet(\LGr(n, V), \Lambda^{n-1}\CU^*)
    \simeq V^{[n-i]},
  \end{equation*}
  and conclude that the spectral sequence associated to the filtered complex~\eqref{eq:fil-wn}
  degenerates into a complex of the form~\eqref{eq:stc-wn}.
\end{proof}

\subsection{Fullness of the Kuznetsov--Polishchuk exceptional collection}
It should be clear from the name of the present section that we are finally
going to prove that the exceptional collection constructed by Kuznetsov and Polishchuk
is full.

\begin{theorem}
  The exceptional collection~\eqref{eq:kp-coll} is full.
\end{theorem}
\begin{proof}
  Let us denote by $\CC$ the subcategory generated by the exceptional
  collection~\eqref{eq:kp-coll}. In order to show that $\CC=D^b(\LGr(n, V))$,
  it is enough to prove that $\CC(1)\subseteq\CC$. Indeed, as $\CO\in\CC$,
  the latter would imply that $\CO(i)\in \CC$ for all $i>0$. It is well known
  that for a smooth projective variety of dimension $d$ the object
  $\oplus_{i=0}^{d}\CO(i)$ is a classical generator of its bounded derived category
  (see~\cite[Theorem~4]{Orl}). Thus, we will be able to conclude that $\CC$ coincides with
  the whole bounded derived category.

  Let $\lambda$ be a Young diagram. Put $w(\lambda)=\lambda_1$,
  and let $h(\lambda)$ denote the number of nonzero rows in $\lambda$.
  The exceptional collection~\eqref{eq:kp-coll}
  consists of the objects $\CE^\lambda(i)$, where $h(\lambda)+w(\lambda) \leq n$,
  and the twist $i$ runs over the integers $h(\lambda), h(\lambda)+1, \ldots, n-w(\lambda)$.
  We will show by induction that $\CE^\lambda(n-w(\lambda)+1)\in \CC$
  for any $\lambda$ such that $h(\lambda)+w(\lambda) \leq n+1$.

  Let $\lambda$ be a Young diagram with $h(\lambda)+w(\lambda) \leq n+1$ and $w(\lambda) > 0$
  (the last condition implies that $\lambda$ is non-zero, we will deal with the latter
  case in the end of the proof). Denote by $t(\lambda)$ the number of rows of width
  $w(\lambda)$. In other words, either $\lambda$ is such that
  $\lambda_1=\lambda_2=\cdots=\lambda_{t(\lambda)}>\lambda_{t(\lambda)+1}$,
  or $t(\lambda)=h(\lambda)$, and $\lambda_1=\lambda_2=\ldots=\lambda_{h(\lambda)}$
  (see Figure~\ref{fig:hwt}).

  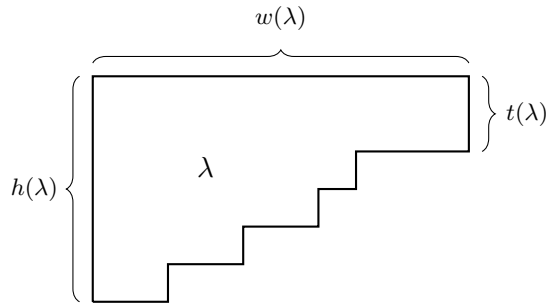
\begin{figure}[h]
    \begin{tikzpicture}
      \draw[thick] (0, 0) -- (0, 3) -- (5, 3) -- (5, 2) -- (3.5, 2) -- (3.5, 1.5)
      -- (3, 1.5) -- (3, 1) -- (2, 1) -- (2, 0.5) -- (1, 0.5) -- (1, 0) -- (0, 0);
      \draw (1.5, 1.8) node{$\lambda$};
      \draw [decorate,decoration={brace,amplitude=5pt},xshift=-5pt,yshift=0pt] (0, 0) -- (0, 3) node [black,midway,xshift=-0.6cm] {\footnotesize $h(\lambda)$};
      \draw [decorate,decoration={brace,amplitude=5pt},xshift=0pt,yshift=5pt] (0, 3) -- (5, 3) node [black,midway,yshift=0.6cm] {\footnotesize $w(\lambda)$};
      \draw [decorate,decoration={brace,amplitude=5pt},xshift=5pt,yshift=0pt] (5, 3) -- (5, 2) node [black,midway,xshift=0.6cm] {\footnotesize $t(\lambda)$};
    \end{tikzpicture}
    \caption{$h(\lambda)$, $w(\lambda)$, and $t(\lambda)$}\label{fig:hwt}
  \end{figure}

  The induction will run on $t(\lambda)$, the base case being $t(\lambda)=1$.
  For convenience, put $w=w(\lambda)$ and $h=h(\lambda)$.
  Let $\lambda$ be such that $t(\lambda) = 1$. We treat $\lambda$ as an element
  of $\You_{n+1-w, w}$.
  Consider the complex~\eqref{eq:stc}
  twisted by $\CO(n-w+1)$:
  \begin{equation}\label{eq:twcpx}
    0\to \CE^{\lambda'}(n-w)\to \SpS{V}{\nu_w}\otimes\CE^{\lambda^{(w)}}(n-w+1)\to \cdots
    \to \SpS{V}{\nu_1}\otimes\CE^{\lambda^{(1)}}(n-w+1)\to
    \CE^\lambda(n-w+1)\to 0.
  \end{equation}
  By definition, $h(\lambda^{(i)})\leq h$ and $w(\lambda^{(i)})\leq w-1$
  for all $i=1,\ldots,w$. In particular, $h(\lambda^{(i)})+w(\lambda^{(i)})\leq n$,
  and $\CE^{\lambda^{(i)}}$ appears in the exceptional collection with twists
  ranging from $h(\lambda^{(i)})\leq h$ to $n-w(\lambda^{(i)})\geq n-w+1$.
  The latter includes $n-w+1$; thus, every term
  of the complex~\eqref{eq:twcpx} of the form $\SpS{V}{\nu_i}\otimes\CE^{\lambda^{(i)}}(n-w+1)$
  belongs to $\CC$. Meanwhile, $w(\lambda')=\lambda_2+1\leq w$ (here we use the hypothesis
  $t(\lambda)=1$, which implies $\lambda_2<\lambda_1=w$), and $h(\lambda')\leq (n+1)-w-1=n-w$
  (see the remark following~\eqref{eq:lprime}).
  Thus, $w(\lambda')+h(\lambda')\leq n$, and the bundle $\CE^{\lambda'}(n-w)$
  belongs to the exceptional collection~\eqref{eq:kp-coll}. Treating~\eqref{eq:twcpx}
  as a resolution for $\CE^\lambda(n-w+1)$, we conclude that $\CE^\lambda(n-w+1)\in\CC$.

  For the inductive step, assume that the statement is known for $0<t(\lambda)\leq t$.
  Let $\lambda$ be such that $t(\lambda)=t+1$. Again, consider the exact complex~\eqref{eq:twcpx}.
  The exact same reasoning as in the base case shows that
  $\SpS{V}{\nu_i}\otimes\CE^{\lambda^{(i)}}(n-w+1)\in\CC$. Meanwhile,
  $t(\lambda)\geq 2$ implies $t(\lambda') = t(\lambda)-1$ and $w(\lambda')=w+1$.
  By the inductive hypothesis $\CE^{\lambda'}(n-w(\lambda')+1)=\CE^{\lambda'}(n-w)$ belongs
  to $\CC$. Again, treating~\eqref{eq:twcpx}
  as a resolution for $\CE^\lambda(n-w+1)$, we conclude that $\CE^\lambda(n-w+1)\in\CC$.

  So far we managed to prove that if $\lambda$ is such that $h(\lambda)+w(\lambda)\leq n$
  and $t(\lambda)\geq 0$, then $\CE^\lambda(n-w+1)\in\CC$. We are left with the case
  $t(\lambda)=0$, which corresponds to $\CO$. Precisely, we need to show that $\CO(n+1)\in\CC$.
  Remark that we managed to show a little bit more. Namely, we showed that
  if the diagram~$\lambda$ is such that $h(\lambda)+w(\lambda)\leq n+1$
  and $t(\lambda)\geq 0$, then $\CE^\lambda(i)\in\CC$ for $i=h, \ldots, n-w+1$.
  Consider the diagram $\mu=(\underbrace{1, 1, \ldots, 1}_{n})$.
  On the one hand, $h(\mu)=n$ and $w(\mu)=1$; thus, $\CE^\mu(n)\in\CC$.
  On the other hand, $\CE^\mu\simeq \CO(1)$; thus, $\CE^\mu(n)\simeq\CO(n+1)\in\CC$.
\end{proof}

\section{Minimal Lefschetz exceptional collection in $D^b(\LGr(5, 10))$}

The exceptional collections of Kuznetsov--Polishchuk may not be the most suitable for
some important computations.
In the present section we construct a minimal Lefschetz exceptional collection
in $D^b(\LGr(5, 10))$.

\subsection{Lefschetz exceptional collections}
Let $X$ be a smooth projective variety, and let $\CO(1)$ be a very ample line bundle on $X$.
\begin{definition}
  A \emph{Lefschetz semiorthogonal decomposition} is a semiorthogonal decomposition
  of the form
  \begin{equation*}
    D^b(X)=\langle \CB_0, \CB_1(1),\ldots, \CB_{r-1}(r-1)\rangle,
  \end{equation*}
  where $\CB_0\supseteq \CB_1\supseteq\cdots\supseteq \CB_{r-1}$ are full triangulated
  subcategories, which are called \emph{blocks}.
  A Lefschetz semiorthogonal decomposition is \emph{minimal} if it is minimal with
  respect to the partial inclusion order on the first block.
\end{definition}
When $X$ is a Fano variety, and $\omega_X\simeq \CO(-r)$, Serre duality implies that
$r$ is the maximal number of blocks in a Lefschetz semiorthogonal decomposition.

Lefschetz decompositions are one of the core components of the Homological Projective duality
theory developed by Kuznetsov, see~\cite{HPD}. One of their most pleasant properties
is the following observation.

\begin{proposition}
  Let $\iota:Y\to X$ be a smooth hyperplane section of $X$ with respect to $\CO(1)$.
  The functor $\iota^*$ is fully faithful on $\CB_i(i)$ for $i=1,\ldots,r-1$;
  moreover, there is a semiorthogonal decomposition
  \begin{equation*}
    D^b(Y)\supseteq \langle \iota^*\CB_1(1), \iota^*\CB_2(2), \ldots, \iota^*\CB_{r-1}(r-1)\rangle.
  \end{equation*}
\end{proposition}

In particular, whenever $\CB_0$ is small, one knows quite a lot of information about
the derived category of hyperplane sections of $X$.

\begin{definition}
  Let $D^b(X)=\langle \CB_0, \CB_1(1),\ldots, \CB_{r-1}(r-1)\rangle$ be a Lefschetz
  semiorthogonal decomposition. If $\CB_0$ is generated by a full
  exceptional collection, and each $\CB_i$ is generated by its subcollection, we will say
  that the resulting exceptional collection is a \emph{Lefschetz exceptional collection}.
  A Lefschetz exceptional collection is \emph{minimal} if the exceptional collection
  generating $\CB_0$ is of the smallest possible length.
\end{definition}

\subsection{Minimal Lefschetz exceptional collection in $D^b(\LGr(5, 10))$}
Consider a $10$-dimensional symplectic vector space $V$ over $\kk$. Our goal is to construct
a minimal Lefschetz exceptional collection in the~derived category $D^b(\LGr(5, V))$.
Since the rank of $K_0(\LGr(5, V))=2^5=32$, and $\omega_{\LGr(5,V)}\simeq \CO(-6)$,
the~smallest possible number of objects in the first block of such a collection equals~5.

\begin{theorem}\label{thm:510}
  The bounded derived category of coherent sheaves on $\LGr(5, V)$ admits
  a full minimal Lefschetz exceptional collection of the form
  \begin{equation}\label{eq:510}
    D^b(\LGr(5, V)) =
    \begin{pmatrix*}[r]
      \CE^{2, 2} & \CE^{2, 2}(1) & & & & \\
      \CE^{2, 1} & \CE^{2, 1}(1) & \CE^{2, 1}(2) & \CE^{2, 1}(3) & \CE^{2, 1}(4) & \CE^{2, 1}(5) \\
      \CE^{2} & \CE^{2}(1) & \CE^{2}(2) & \CE^{2}(3) & \CE^{2}(4) & \CE^{2}(5) \\
      \Lambda^2\CU^* & \Lambda^2\CU^*(1) & \Lambda^2\CU^*(2) & \Lambda^2\CU^*(3) & \Lambda^2\CU^*(4) & \Lambda^2\CU^*(5) \\
      \CU^* & \CU^*(1) & \CU^*(2) & \CU^*(3) & \CU^*(4) & \CU^*(5) \\
      \CO & \CO(1) & \CO(2) & \CO(3) & \CO(4) & \CO(5) \\
    \end{pmatrix*}.
  \end{equation}
\end{theorem}

\begin{proposition}\label{prop:510exc}
  The collection of vector bundles~\eqref{eq:510} is exceptional.
\end{proposition}
\begin{proof}
  Remark that the objects in the first column all belong to the exceptional
  block $\langle \CE^\lambda \mid \lambda\in \You_{2, 2}\rangle$.
  Thus, the collection is exceptional in each column.
  It remains to check that
  \begin{equation*}
    \Ext^\bullet(\CE^\lambda(t), \CE^\mu) = 0
  \end{equation*}
  for $0\subseteq \lambda \subseteq (2, 1)$, $0\subseteq \mu \subseteq (2, 2)$, $t=1,\ldots,5$,
  and $\lambda=(2,2)$, $0\subseteq \mu \subseteq (2, 2)$, $t=1$.

  Since each of the objects involved is an extension of irreducible vector bundles
  of the form $\Sigma^\nu\CU^*$, where~$\nu\in\You_{2, 2}$, it is enough to check that
  \begin{equation}\label{eq:510cu}
    \Ext^\bullet(\Sigma^\lambda\CU^*(t), \Sigma^\mu\CU^*) = 
    H^\bullet(\LGr(5, 10), \Sigma^\lambda\CU\otimes\Sigma^\mu\CU^*(-t)) = 0.
  \end{equation}
  for $0\subseteq \lambda \subseteq (2, 1)$, $0\subseteq \mu \subseteq (2, 2)$, $t=1,\ldots,5$,
  and $\lambda=(2,2)$, $0\subseteq \mu \subseteq (2, 2)$, $t=1$.
  By Lemma~\ref{lm:prod}, $\Sigma^\lambda\CU\otimes\Sigma^\mu\CU^*$ decomposes
  into a direct sum of irreducible bundles of the form $\Sigma^\nu\CU^*$, where
  $-\lambda\subseteq \nu \subseteq \mu$. Thus, it is sufficient to check that for any
  $\nu=(\alpha, \beta, 0, \gamma, \delta)\in\You_5$ such that $0\leq \alpha, \beta, \gamma, \delta\leq 0$
  \begin{equation*}
    H^\bullet(\LGr(5, 10), \Sigma^\nu\CU^*(-t)) = H^\bullet(\LGr(5, 10), \Sigma^{\nu}\CU^* = 0
  \end{equation*}
  if $t=1,\ldots,5$ and $\nu\neq (2,2,0,-2,-2)$, or
  $\nu=(2,2,0,-2,-2)$ and $t=1, 3, 5$.

  By the Borel--Bott--Weil theorem, we need to look at the weight
  \begin{equation}\label{eq:510bbw}
    \rho + \nu(-t) = (5+\alpha-t, 4+\beta-t, 3-t, 2+\gamma-t, 1+\delta-t)
  \end{equation}
  By Serre duality, it is enough to deal with $t=1,2,3$. When $t=3$, the third term
  in~\eqref{eq:510bbw} equals 0; thus, the cohomology groups vanish.
  When $t=1$, the last three terms in the sequence have absolute values at most 2.
  Thus, either one of them is 0, or the absolute values of a pair of them are equal.
  Similarly, when $t=2$, we see that the only option for the absolute values of the terms
  of~\eqref{eq:510bbw} to be positive and distinct is $\nu=(2, 2, 0, -2, -2)$.
\end{proof}

\begin{remark}
  One can continue the computation in the proof of Proposition~\ref{prop:510exc}
  and conclude that
  \begin{equation*}
    \Ext^5(\CE^{2, 2}(2), \CE^{(2, 2)})\simeq \kk.
  \end{equation*}
  A nontrivial element it this group can be realized
  as the Yoneda product of the~(twisted) staircase complexes
  \begin{equation*}
    0\to \CE^{3, 1, 1}(1)\to V^{[3]}\otimes \CU^*(2)\to V^{[2]}\otimes \Lambda^2\CU^*(2)
    \to \CE^{2, 2}(2)\to 0
  \end{equation*}
  and
  \begin{equation*}
    0\to \CE^{2,2}\to V^{[5]}\otimes\CO(1) \to V^{[2]}\otimes \Lambda^3\CU^*(1) \to
    V\otimes\CE^{2,1,1}(1) \to \CE^{3,1,1}(1)\to 0.
  \end{equation*}
  The resulting complex is
  \begin{multline}\label{eq:22ext}
    0\to \CE^{2,2}\to V^{[5]}\otimes\CO(1) \to V^{[2]}\otimes \Lambda^3\CU^*(1) \to
    V\otimes\CE^{2,1,1}(1) \to \\
    \to V^{[3]}\otimes \CU^*(2)\to V^{[2]}\otimes \Lambda^2\CU^*(2)
    \to \CE^{2, 2}(2)\to 0.
  \end{multline}
\end{remark}

\begin{proposition}\label{prop:510full}
  The exceptional collection~\eqref{eq:510} is full.
\end{proposition}
\begin{proof}
  Let $\CT$ denote the full triangulated subcategory generated by the exceptional
  collection~\eqref{eq:510}. By Theorem~\ref{thm:meta}, it is enough to show that
  $\CT$ contains all the objects from the full exceptional collection
  \begin{equation*}
    D^b(\LGr(5, 10)) = \langle \CB_0, \CB_1, \CB_2, \CB_3, \CB_4, \CB_5\rangle,
  \end{equation*}
  where
  \begin{align}\label{eq:510ce}
    \begin{split}
    \CB_0 & = \langle \CO \rangle, \\
    \CB_1 & = \langle \CO(1),\ \CU^*(1),\ \CE^{2}(1),\ \CE^3(1),\ \CE^4(1) \rangle, \\
    \CB_2 & = \langle \CO(2),\ \CU^*(2),\ \Lambda^2\CU^*(2),\ \CE^{2}(2),\ \CE^{2, 1}(2),\ \CE^{2, 2}(2),\ \CE^{3}(2),\ \CE^{3,1}(2),\ \CE^{3,2}(2), \CE^{3,3}(2)\rangle, \\
    \CB_3 & = \langle \CO(3),\ \CU^*(3),\ \Lambda^2\CU^*(3),\ \CE^{2}(3),\ \CE^{2, 1}(3),\ \CE^{2, 2}(3),\ \Lambda^3\CU^*(3),\ \CE^{2, 1, 1}(3),\ \CE^{2, 2, 1}(3),\ \CE^{2, 2, 2}(3)\rangle, \\
    \CB_4 & = \langle \CO(4),\ \CU^*(4),\ \Lambda^2\CU^*(4),\ \Lambda^3\CU^*(4),\ \Lambda^4\CU^*(4) \rangle, \\
    \CB_5 & = \langle \CO(5) \rangle.
    \end{split}
  \end{align}
  Quite a number of the objects from~\eqref{eq:510ce} are trivially in $\CT$. We need to deal with
  the remaining ones:
  \begin{align}\label{eq:510ntriv}
    \begin{split}
      &\CE^3(1),\ \CE^4(1), \\
      &\CE^{3}(2),\ \CE^{2, 2}(2),\ \CE^{3,1}(2),\ \CE^{3,2}(2), \CE^{3,3}(2), \\
      &\Lambda^3\CU^*(3),\ \CE^{2, 2}(2),\ \CE^{2, 1, 1}(3),\ \CE^{2, 2, 1}(3),\ \CE^{2, 2, 2}(3), \\
      &\Lambda^3\CU^*(4),\ \Lambda^4\CU^*(4). \\
    \end{split}
  \end{align}
  The proof is split into steps. At each step we write some exact complex (built of staircase
  complexes), all but one of whose terms are already known to be in $\CT$ (from the previous
  steps).

  \textit{Step 1:} $\CE^{3}(t)\in\CT$ for $t=1, \ldots, 5$. Enough to look at the twisted
  staircase complex
  \begin{equation*}
    0\to \Lambda^2\CU^*(t-1)\to V^{[3]}\otimes\CO(t)\to V^{[2]}\otimes\CU^*(t)\to
    V\otimes\CE^{2}(t)\to \CE^{3}(t)\to 0.
  \end{equation*}

  \textit{Step 2:} $\CE^{4}(t)\in\CT$ for $t=1, \ldots, 5$. Enough to look at the twisted
  staircase complex
  \begin{equation*}
    0\to \CU^*(t-1)\to V^{[4]}\otimes\CO(t)\to V^{[3]}\otimes\CU^*(t)\to
    V^{[2]}\otimes\CE^{2}(t)\to V\otimes\CE^{3}(t)\to \CE^{4}\to 0.
  \end{equation*}

  \textit{Step 3:} $\Lambda^3\CU^*(t)\in\CT$ for $t=0, \ldots, 4$. Enough to look at the twisted
  staircase complex
  \begin{equation*}
    0\to \Lambda^3\CU^*(t)\to V^{[2]}\otimes\CO(t+1)\to
    V\otimes\CU^*(t+1)\to \CE^{2}(t+1)\to 0.
  \end{equation*}

  \textit{Step 4:} $\Lambda^4\CU^*(t)\in\CT$ for $t=0, \ldots, 4$. Since $\Lambda^4\CU^*\simeq \CU(1)$,
  Enough to look at the twisted tautological short exact sequence
  \begin{equation*}
    0\to \Lambda^4\CU^*(t)\simeq \CU(t+1)\to V\otimes\CO(t+1)\to \CU^*(t+1)\to 0.
  \end{equation*}
  (Which also happens to be a staircase complex.)

  \textit{Step 5:} $\CE^{2, 1, 1}(t)\in\CT$ for $t=0, \ldots, 4$. Enough to look at the twisted
  staircase complex
  \begin{equation*}
    0\to \CE^{2, 1, 1}(t)\to V^{[3]}\otimes\CO(t+1)\to
    V\otimes\Lambda^2\CU^*(t+1)\to \CE^{2,1}(t+1)\to 0.
  \end{equation*}

  \textit{Step 6:} $\CE^{2, 2}(t)\in\CT$ for $t=0, \ldots, 5$. For $t=2,\ldots,5$, it is
  enough to look at the
  complex~\eqref{eq:22ext} twisted by~$\CO(t-2)$.

  \textit{Step 7:} $\CE^{3, 1}(t)\in\CT$ for $t=1, \ldots, 5$. Enough to look at the twisted
  staircase complex
  \begin{equation*}
    0\to \CE^{2,1}(t-1)\to V^{[4]}\otimes\CO(t)\to V^{[2]}\otimes\Lambda^2\CU^*(t)\to
    V\otimes\CE^{2,1}(t)\to \CE^{3,1}(t)\to 0.
  \end{equation*}

  \textit{Step 8:} $\CE^{2, 2, 1}(t)\in\CT$ for $t=0, \ldots, 3$. Enough to look at the twisted
  staircase complex
  \begin{equation*}
    0\to \CE^{2, 2, 1}(t)\to V^{[4]}\otimes\CO(t+1)\to
    V\otimes\Lambda^3\CU^*(t+1)\to \CE^{2,1,1}(t+1)\to 0.
  \end{equation*}

  \textit{Step 9:} $\CE^{3, 2}(t)\in\CT$ for $t=2, \ldots, 5$. Enough to look at the twisted
  staircase complex
  \begin{equation*}
    0\to \CE^{3,1}(t-1)\to V^{[5]}\otimes\CU^*(t)\to V^{[3]}\otimes\Lambda^2\CU^*(t)\to
    V\otimes\CE^{2,2}(t)\to \CE^{3,2}(t)\to 0.
  \end{equation*}

  \textit{Step 10:} $\CE^{2, 2, 2}(t)\in\CT$ for $t=0,\ldots,3$. Enough to look at the Yoneda
  product of the twisted staircase complexes
  \begin{equation*}
    0\to \CE^{2,2,2}(t)\to V^{[5]}\otimes\CO(t+1)\to V\otimes\Lambda^4\CU^*(t+1)\to \CE^{2, 1, 1, 1}(t+1)\to 0
  \end{equation*}
  and
  \begin{equation*}
    0\to \CE^{2, 1, 1, 1}(t+1)\to V^{[2]}\otimes\CO(t+2)\to \Lambda^2\CU^*(t+2)\to 0.
  \end{equation*}

  \textit{Step 11:} $\CE^{3, 3}(t)\in\CT$ for $t=2,\ldots,5$. Enough to look at the Yoneda
  product of the twisted staircase complexes
  \begin{equation*}
    0\to \CE^{4,1}(t-1)\to V^{[4]}\otimes\CE^{2}(t)\to V^{[3]}\otimes\CE^{2,1}(t)\to V^{[2]}\otimes\CE^{2,2}(t)\to
    \CE^{3,3}(t)\to 0
  \end{equation*}
  and
  \begin{equation*}
    0\to \CE^{2}(t-2)\to V^{[5]}\otimes\CO(t-1)\to V^{[3]}\otimes\Lambda^2\CU^*(t-1)\to
    V^{[2]}\otimes\CE^{2,1}(t-1)\to V\otimes\CE^{3,1}(t-1)\to \CE^{4,1}(t-1)\to 0.
  \end{equation*}

  It follows from Steps 1 through 11 that all the objects from~\eqref{eq:510ntriv} belong to $\CT$.
\end{proof}

\begin{proof}[Proof of Theorem~\ref{thm:510}]
  Follows from Propositions~\ref{prop:510exc} and~\ref{prop:510full}.
\end{proof}

\begin{remark}\label{rm:pol-sam}
  In the paper~\cite{Sam-Pol} the authors construct a exceptional collection
  in $D^b(\LGr(5, 10))$, which is very close to being minimal Lefschetz.
  Unfortunately, it was recently discovered by M.~Smirnov that their collection
  is not exceptional.
\end{remark}

\appendix

\section{Cohomological computations}

\subsection{Borel--Bott--Weil theorems}
Here we collect individual statements, all of which are particular cases
of the celebrated Borel--Bott--Weil theorem. We present the most concrete
statements.
Proofs for all of them can be found in the excellent book~\cite{Wey}.

In the following, given two sequences
$\alpha = (\alpha_1,\alpha_2,\ldots,\alpha_a)\in\ZZ^a$ and
$\beta = (\beta_1,\beta_2,\ldots,\beta_b)\in\ZZ^b$, we denote by $(\alpha,\beta)$
their concatenation
\[
  (\alpha,\beta)=(\alpha_1,\alpha_2,\ldots,\alpha_a,
  \beta_1,\beta_2,\ldots,\beta_b)\in\ZZ^{a+b},
\]

Let us begin with the relative classical Grassmannian case.

\begin{proposition}[Relative Borel--Bott--Weil]\label{prop:bbwrel}
  Let $X$ be a smooth projective variety, let $\CV$ be a rank $n$ vector bundle
  on $X$, and let $0<k<n$ be an integer. Consider the relative Grassmannian
  \begin{equation*}
    p: \Gr_X(k, \CV)\to X.
  \end{equation*}
  We denote the pullback of $\CV$ on $\Gr(k, \CV)$ by the same letter,
  and denote by $\CU\subset \CV$ the universal rank $k$ subbundle on $\Gr_X(k, \CV)$.

  Let $\lambda\in\You_{n-k}$, $\mu\in\You_k$, and put $\rho=(n, n-1, \ldots, 1)$. Then
  \begin{equation*}
    R^\bullet p_*\left(\Sigma^\lambda\!\left(\CV/\CU\right)\otimes \Sigma^\mu\CU\right) = 
    \begin{cases}
      \Sigma^{w\cdot(\lambda,\mu)}\CV [-\ell(w)], & \text{if all the elements in } \rho+(\lambda,\mu) \text{ distinct,}\\
      0, & \text{otherwise,}
    \end{cases}
  \end{equation*}
  where $w\in\mathfrak{S}_n$ denotes the unique permutation such that the sequence
  $w((\alpha,\beta)+\rho)$ is strictly decreasing,
  $w\cdot (\lambda,\mu) = w((\rho+(\lambda,\mu))-\rho$, and $\ell(w)$ is the number
  of pairs $1\leq i<j\leq n$ such that $w(i)>w(j)$.
\end{proposition}

\begin{proposition}[Isotropic Borel--Bott--Weil]\label{prop:bbwiso}
  Let $V$ be a $2n$-dimensional symplectic vector space, and let $0<w<n$ be an integer.
  Consider the isotropic Grassmannian $\IGr(w, V)$, and denote by $\CW$ the~tautological
  rank $w$ vector bundle. Let us say that a weight $\lambda\in \You_n$ is regular
  if all the absolute values $|\lambda_i|$ are positive and distinct.
  Let $\rho=(n, n-1, \ldots, 1)\in \You_n$. We also consider the action
  of the group $\mathfrak{S}_n \ltimes \ZZ_2^n$ on $\You_n$ by permutations and sign changes.
  
  Given $\alpha\in\You_w$ and $\beta\in\You_{n-w}$,
  \begin{equation*}
    H^\bullet(\Sigma^\alpha\CW^*\otimes\SpS{(\PerpQ{\CW})}{\beta}) = \begin{cases}
      \SpS{V}{w\cdot (\alpha,\beta)}[-\ell(w)] & \text{if } \rho+(\alpha,\beta) \text{ is regular}, \\
      0 & \text{otherwise},
    \end{cases}
  \end{equation*}
  where $w\in \mathfrak{S}_n \ltimes \ZZ_2^n$ is the unique element such that the sequence
  $w(\rho+(\alpha, \beta))$ is positive and strictly decreasing, and
  $w\cdot (\alpha,\beta) = w((\rho+(\alpha,\beta))-\rho$.
  The number $\ell(w)$ can be computed from $\tau=\rho+(\alpha,\beta)$
  as the number of pairs $1\leq i<j\leq n$ such that $\tau_i < \tau_j$ plus
  the number of pairs $1\leq i<j\leq n$ such that $\tau_i + \tau_j < 0$ plus
  the number of negative elements in $\tau$.
\end{proposition}

\begin{proposition}[Relative Lagrangian Borel--Bott--Weil]\label{prop:bbwrellg}
  Let $X$ be a scheme, and let $\CV$ be a rank $2n$ symplectic vector bundle on $X$.
  Consider the relative Lagrangian Grassmannian $p:\LGr(n, \CV)\to X$, and let
  $\CU$ denote the relative tautological bundle of rank $n$. For any $\lambda\in\You_n$
  one has
  \begin{equation*}
    R^\bullet p_*\Sigma^\lambda\CU^* = \begin{cases}
      \SpS{\CV}{w\cdot \lambda}[-\ell(w)] & \text{if } \rho+\lambda \text{ is regular}, \\
      0 & \text{otherwise},
    \end{cases}
  \end{equation*}
  where $w\in \mathfrak{S}_n \ltimes \ZZ_2^n$ is the unique element such that the sequence
  $w(\rho+\lambda)$ is positive and strictly decreasing, and $\ell(w)$ can be computed from
  $\tau=\rho+\lambda$ as in the previous proposition.
\end{proposition}

\subsection{Borel--Bott--Weil computations}

The following lemmas are trivial; we include their proofs for the sake of completeness.

\begin{lemma}\label{lm:gr-0}
  Let $X$ be a scheme, let $\CV$ be a rank $n$ vector bundle on $X$, and let $0<k<n$
  be an integer.
  Consider the relative Grassmannian $p:\Gr(k, \CV)\to X$, and denote by $\CU$ the tautological
  rank $k$ vector bundle on $\Gr(k, V)$.
  If $\lambda\in\You_{n-k}$ is such $\lambda_{n-k}\geq -k$, then
  \begin{equation*}
    R^ip_* \Sigma^\lambda(\CV/\CU) = \begin{cases}
      \Sigma^\lambda\CV & \text{if } \lambda_{n-k}\geq 0 \text{ and } i = 0, \\
      0 & \text{otherwise.}
    \end{cases}
  \end{equation*}
\end{lemma}
\begin{proof}
  According to the relative Borel--Bott--Weil theorem, we need to study the sequence
  \begin{equation}\label{eq:gr-0}
    \rho+(\lambda, 0) = (n+\lambda_1, n-1+\lambda_2, \ldots, k+1+\lambda_{n-k}, k, k-1, \ldots, 1).
  \end{equation}
  If $0>\lambda_{n-k}\geq -k$, then the $(n-k)$-th term of~\eqref{eq:gr-0} satisfies
  the inequality $k\geq k+1+\lambda_{n-k}\geq 1$, and equals one of the last $k$ terms
  of the sequence. Otherwise, $\lambda_{n-k}\geq 0$, and the sequence~\eqref{eq:gr-0}
  is strictly decreasing.
\end{proof}

\begin{lemma}\label{lm:lgr-0}
  Let $X$ be a scheme and let $\CV$ be a symplectic vector bundle on $X$ of rank $2n$.
  Consider the relative Lagrangian Grassmannian $p:\LGr(n, \CV)\to X$, and denote by 
  $\CU$ the tautological subbundle on it. If $\lambda\in \You_n$ is such that $\lambda_n\geq -1$,
  then
  \begin{equation*}
    R^ip_* \Sigma^\lambda\CU^* = \begin{cases}
      \SpS{\CV}{\lambda} & \text{if } \lambda_n\geq 0 \text{ and } i = 0, \\
      0 & \text{otherwise.}
    \end{cases}
  \end{equation*}
\end{lemma}
\begin{proof}
  According to the relative Lagrangian Borel--Bott--Weil theorem, we need to study the sequence
  \begin{equation}\label{eq:lgr-0}
    \rho+\lambda = (n+\lambda_1, n-1+\lambda_2, \ldots, 2+\lambda_2, 1+\lambda_1).
  \end{equation}
  If $\lambda_1=-1$, then the last term of~\eqref{eq:gr-0} equals $0$; thus, all the
  direct images vanish. Otherwise, the sequence~\eqref{eq:gr-0} consists of
  strictly decreasing positive numbers.
\end{proof}

Let $V$ be a $2n$-dimensional symplectic
vector space. Consider the isotropic Grassmannian $\IGr(w, V)$ for some $0<w\leq n$,
and denote by $\CW$ the tautological bundle on it.

\begin{lemma}\label{lm:igr-van}
  Let $\nu\in\You_w$ be such that $\nu_w\geq -(2n-2w+1)$. Then
  \begin{equation*}
    H^i(\IGr(w, V),\, \Sigma^\nu \CU^*) =
    \begin{cases}
      \SpS{V}{\nu} & \text{if } \nu_w \geq 0 \text{ and } i = 0, \\
      0 & \text{otherwise}.
    \end{cases}
  \end{equation*}
\end{lemma}
\begin{proof}
  According to the relative Lagrangian Borel--Bott--Weil theorem, we need to study the sequence
  \begin{equation}\label{eq:igr-van}
    \rho+(\nu, 0) = (n+\nu_1,\ n-1+\nu_2,\ \ldots,\ n-w+1+\nu_{w},\ n-w,\ n-w-1, \ldots,\ 1).
  \end{equation}
  If $0>\nu_{w}\geq -(2n-2w+1)$, then the $w$-th term of~\eqref{eq:igr-van} satisfies
  $n-w\geq n-w+1+\nu_{w}\geq -(n-w)$. In particular, it is either zero,
  or its absolute value coincides with one the the last $n-w$ terms of~\eqref{eq:igr-van}.
  Otherwise, $\nu_{w}\geq 0$, and the sequence~\eqref{eq:igr-van} is strictly decreasing.
\end{proof}

\begin{lemma}\label{lm:igr-0}
  If $\alpha\in\YD_{n-w}$ and $\beta\in\YD_w$, then
  \begin{equation}
    H^\bullet(\IGr(w, 2n), \SpS{(\PerpQ{\CW})}{\alpha}\otimes\Sigma^\beta\CW^*)^\bfG =
    \begin{cases}
      \kk & \text{if }\alpha=\beta=0, \\
      0 & \text{otherwise.}
    \end{cases}
  \end{equation}
\end{lemma}
\begin{proof}
  Using the Borel--Bott--Weil theorem, it is enough to check when the absolute
  values of the sequence
  \begin{equation*}
    \rho+(\beta,\alpha) = (n+\beta_1,\ n-1+\beta_2,\ \ldots,\ n-w+1+\beta_w,\ n-w+\alpha_1,\ \ldots,\ 1+\alpha_{n-k})
  \end{equation*}
  reordered in decreasing order coincide with $\rho$. Since all $\alpha_i$ and $\beta_j$
  are nonnegative, it can only happen when they are all equal to zero, in which case
  $\rho+(\beta,\alpha)=\rho$.
\end{proof}

\begin{lemma}\label{lm:igr-eq}
  If $\alpha,\beta\in \YD_{n-w}$, then
  \begin{equation*}
    \Ext_{\bfG}(\SpS{(\PerpQ{\CW})}{\alpha}, \SpS{(\PerpQ{\CW})}{\beta}) = \begin{cases}
      \kk & \text{if } \alpha = \beta, \\
      0 & \text{otherwise}.
    \end{cases}
  \end{equation*}
\end{lemma}
\begin{proof}
  Since $\PerpQ{\CW}$ is self-dual,
  \begin{align*}
    \Ext_{\bfG}(\SpS{(\PerpQ{\CW})}{\alpha}, \SpS{(\PerpQ{\CW})}{\beta})
    & \simeq H^\bullet(\IGr(w, 2n), \SpS{(\PerpQ{\CW})}{\alpha}\otimes\SpS{(\PerpQ{\CW})}{\beta})^\bfG \\
    & \simeq \bigoplus H^\bullet(\IGr(w, 2n), \SpS{(\PerpQ{\CW})}{\nu})^\bfG,
  \end{align*}
  where $\nu$ runs over the irreducible direct summands in the decomposition
  \[
    \SpS{(\PerpQ{\CW})}{\alpha}\otimes\SpS{(\PerpQ{\CW})}{\beta} \simeq \bigoplus\, \SpS{(\PerpQ{\CW})}{\nu}.
  \]
  By the previous lemma, $H^\bullet(\IGr(w, 2n), \SpS{(\PerpQ{\CW})}{\nu})^\bfG=0$ unless $\nu=0$.
  In the latter case we immediately conclude that $\SpS{(\PerpQ{\CW})}{\alpha}\simeq\SpS{(\PerpQ{\CW})}{\beta}$.
  In particular, $\alpha=\beta$, and the claim follows trivially.
\end{proof}

Let us now recall one of the main computations in~\cite{Kap}.
We present it in the relative case.

\begin{lemma}[{\cite[Lemma~3.2]{Kap}}]\label{lm:gr-kap}
  Let $X$ be a scheme, let $\CU$ be a rank $n$ vector bundle on $X$, and let
  $0<w<n$ be an integer.
  Consider the relative Grassmannian $p:\Gr(w, \CU)$, and denote by $\CW$
  the~tautological rank $w$ vector bundle on $\Gr(w, \CU)$. For any
  $\lambda\in\You_{w, n-w}$ and $\mu\in\You_{n-w, w}$ one has
  \begin{equation*}
    R^\bullet p_*\left(\Sigma^\lambda\CW\otimes\Sigma^\mu(\CU/\CW)^*\right) = 
    \begin{cases}
      \CO_X[-|\lambda|] & \text{if }\lambda=\mu^T, \\
      0 & \text{otherwise.}
    \end{cases}
  \end{equation*}
\end{lemma}

The following lemma could be proved by a computation similar to the one done in~\cite{Kap}.
For the sake of simplicity, we present a more geometric computation.

\begin{lemma}\label{lm:igr-kap}
  Assume $0<w<n$.
  For any $\lambda\in\You_{w, n-w}$ and $\mu\in\You_{n-w, w}$ one has
  \begin{equation*}
    H^\bullet\left(\IGr(w, V), \Sigma^\lambda\CW\otimes\SpS{(\PerpQ{\CW})}{\mu}\right) = 
    \begin{cases}
      \kk[-|\lambda|] & \text{if }\lambda=\mu^T, \\
      0 & \text{otherwise.}
    \end{cases}
  \end{equation*}
\end{lemma}
\begin{proof}
  Consider the commutative diagram
  \begin{equation*}
    \begin{tikzcd}
      & \IFl(w, n; V) \arrow[ld, "p" swap] \arrow[rd, "q"] & \\
      \LGr(n, V) & & \IGr(w, V).
    \end{tikzcd}
  \end{equation*}
  From Lemma~\ref{lm:lgr-0} and the projection formula we have
  \begin{align*}
    H^\bullet(\IGr(w, V), \Sigma^\lambda\CW\otimes\SpS{(\PerpQ{\CW})}{\mu})
    & \simeq  H^\bullet(\IGr(w, V), \Sigma^\lambda\CW\otimes q_*\Sigma^\mu(\CU/\CW)^*) \\
    & \simeq  H^\bullet(\IFl(w, n; V), \Sigma^\lambda\CW\otimes \Sigma^\mu(\CU/\CW)^*) \\
    & \simeq  H^\bullet(\LGr(n, V), p_*(\Sigma^\lambda\CW\otimes \Sigma^\mu(\CU/\CW)^*)
  \end{align*}
  Since $\IFl(w, n; V)$ is isomorphic to the relative Grassmannian $\Gr(w, \CU)$,
  the result now follows from Lemma~\ref{lm:gr-kap} and Kodaira vanishing.
\end{proof}

For convenience let us put $h=n+1-w$.

\begin{lemma}\label{lm:q-neg}
  Consider the projection $q:\IFl(w, n; V)\to \IGr(w, V)$, and let
  $\alpha$ be a Young diagram such that $\alpha_1\leq h$. Then
  either $q_*\Sigma^\alpha(\CU/\CH) \simeq 0$ or $q_*\Sigma^\alpha(\CU/\CH) \simeq \CO[t]$
  for some $t\in\ZZ$.
\end{lemma}
\begin{proof}
  Remark that $\IFl(w, n; V)$ is the relative Lagrangian Grassmannian
  $\LGr(h-1, \PerpQ{\CW})$.
  According to the Borel--Bott--Weil theorem, we need to study the weight
  \begin{equation}\label{eq:q-neg}
    \rho-\alpha = (h-1-\alpha_{h-1},\, n-2-\alpha_{h-2},\, \ldots,\, 1-\alpha_1).
  \end{equation}
  Since all the absolute values of the elements of~\eqref{eq:q-neg} are at most $h-1$,
  and there are exactly $h-1$ terms, the result follows from pigeonhole principle.
\end{proof}

The following proposition is the most technical (and crucial).
Its proof is very similar to the core computation done
in~\cite[Proposition~5.3]{Fon-LD}. We use the notation introduced in Section~\ref{ssec:comb}.
For simplicity we put $\lambda^{(0)}=\lambda$ and $\nu_0=0$.

\begin{proposition}\label{prop:main}
  Let $\lambda\in\You_{h, w}$ be such that $\lambda_1=w$. Then for all $\mu\in\You_{w, h}$
  \begin{equation*}
    H^\bullet(\IGr(w, V),\, \Sigma^{\mu(-1)}\CW\otimes \SpS{(\PerpQ{\CW})}{\bar{\lambda}}) = 
    \begin{cases}
      V^{[\nu_i]}[-|\lambda^{(i)}|+(w-i)] & \text{if } \mu^T=\lambda^{(i)}, i\in\{0,\ldots,w\},\\
      0 & \text{otherwise}.
    \end{cases}
  \end{equation*}
\end{proposition}
\begin{proof}
  Put $\alpha=\lambda^T$. Let us first point out that $(\lambda^{(i)})^T$ is a little
  easier to describe than $\lambda^{(i)}$ itself:
  \begin{equation}\label{eq:lit}
    (\lambda^{(i)})^T = (\alpha_1,\,\alpha_2,\,\ldots,\,\alpha_{w-i},\,
    \alpha_{w-i+2}-1,\,\ldots,\,\alpha_w-1,\, 0).
  \end{equation}
  In particular, $\nu_i=|\lambda|-|\lambda^{(i)}|=\alpha_{w-i+1}+(i-1)$.

  According to the Borel--Bott--Weil theorem, we need to look at the weight
  \begin{equation}\label{eq:main}
    \rho+(-\mu(-1), \bar{\lambda})=
    (n-\mu_w+1,\, n-1-\mu_{w-1}+1,\,\ldots,\, h-\mu_1+1,\,
     h-1+\lambda_2,\,\ldots,\,2+\lambda_{h-1},\,1+\lambda_h).
  \end{equation}
  Since, $0\leq \mu_i\leq h$, and $0\leq \lambda_j\leq w$, all the terms of~\eqref{eq:main}
  are positive, and their absolute values belong to the set $\{n+1,n,\ldots,1\}$.
  Thus, the weight~\eqref{eq:main} is regular if an only if all its terms are distinct.
  If it is regular, $\ell(w)$ equals the number of inversions in~\eqref{eq:main}

  Assume it is regular. Then the set of its terms must coincide with
  $\{n+1,n,\ldots,1\}\setminus\{t\}$ for some $1\leq t\leq n+1$. In particular,
  $t$ must belong to $\{n+1,n,\ldots,1\}\setminus\{h-1+\lambda_2,\ldots,2+\lambda_{h-1},1+\lambda_h\}$.
  Kapranov showed in~\cite[Lemma~3.2]{Kap} that
  \[
    \{n,\ldots,1\}\setminus\{h-1+\lambda_2,\ldots,2+\lambda_{h-1},1+\lambda_h\} =
    \{n-\beta_w, n-1-\beta_{w-1},\ldots,h-\beta_1\},
  \]
  where $\beta=\bar{\lambda}^T$. Thus, $t\in \{n+1, n-\beta_w, n-1-\beta_{w-1},\ldots,h-\beta_1\}$.

  Assume first that $t=n+1$. Since the first $w$ terms of the sequence~\eqref{eq:main} are
  strictly decreasing, we conclude that $n+2-i-\mu_{w+1-i} = n+1-i-\beta_{w+1-i}$ for
  $i=1,\ldots, w$. In other words, $\mu = \beta(1)$. Since
  $\bar{\lambda}^T=\lambda^T(-1)$ (here we use the condition $\lambda_1=w$),
  we conclude that $\mu=\lambda^T$. Again, Kapranov showed that the number of inversions
  in such case equals $|\bar{\lambda}|=|\lambda|-w$.

  Assume now that $t=n+i-\beta_{w-i}$. Since $\lambda_2\leq w$, the only term which
  can equal $n+1$ is the first one. Thus, $\mu_w=0$.
  Again, since the first $w$ terms of the sequence~\eqref{eq:main} are
  strictly decreasing, we conclude that
  \begin{equation*}
    n+1-j-\mu_{w-j} = n+1-j-\beta_{w+1-j},\quad \text{for }j=1,\ldots,i-1,
  \end{equation*}
  and 
  \begin{equation*}
    n+1-j-\mu_{w-j} = n-j-\beta_{w-j},\quad \text{for }j=i,\ldots,w-1,
  \end{equation*}
  We conclude that
  \begin{equation*}
    \mu = (\beta_1+1,\, \beta_2+1,\,\ldots,\, \beta_{w-i}+1,\, \beta_{w-i+2},\, \ldots,\, \beta_w, 0).
  \end{equation*}
  Since $\beta(1)=\lambda^T$, we see from~\eqref{eq:lit} that $\mu^T=\lambda^{(i)}$.
  Moreover, once put in decreasing order the sequence~\eqref{eq:main} becomes
  \((n+1, n, \ldots, n-i+2+\beta_{w-i+1}, n-i+\beta_{w-i+1}, \ldots, 1)\). Thus,
  $w\cdot (\mu,\bar{\lambda}) = (\nu_i)^T$.

  It remains to compute the number of inversions. It obviously equals the number of inversions
  when $\mu$ equals $\beta$ minus the number of inversions involving $n-i+1-\beta_{w-i+1}$.
  Out of the set $\{n-i-\beta_{w-i+1}, \ldots, 1\}$ exactly $w-i$ elements are coming
  from the~terms $n+1-j-\mu_{w-j}$ for $j=i,\ldots,w-1$. Thus, $h-1-\beta_{w-i+1}$ are
  taken by the elements of the~form $h-j+\lambda_{j+1}$ for $j=1,\ldots,h-1$.
  The remaining $\beta_{w-i+1}=\alpha_{w-i+1}-1$ provide the disappearing inversions.
  We conclude that the total number of inversions equals
  \begin{equation*}
    |\bar{\lambda}|-\beta_{w-i+1} = |\lambda|-(w-i)-\nu_i = |\lambda^{(i)}|-(w-i).
  \end{equation*}
\end{proof}

\bibliographystyle{alpha}
\bibliography{lgr}

\end{document}